\documentclass[aap,MSNbibl,nameyear,dvips]{arximspdf}
\usepackage{mathbh}
\usepackage{graphicx}

\doi{10.1214/13-AAP998} %kopijuoti is PTS
\volume{25}
\issue{1}
\pubyear{2015}
\firstpage{324}
\lastpage{372}

\makeatletter
\newcommand{\dd}{\mathrm{d}}
\newcommand{\rrvert}{\vert}
\newcommand{\llvert}{\vert}
\newcommand{\implies}{\Longrightarrow}
\newtheorem{theor}{Theorem}[section]
\newtheorem{Lem}[theor]{Lemma}
\newtheorem{Prop}[theor]{Proposition}
\newtheorem{cor}[theor]{Corollary}
\newproclaim{defin}[theor]{Definition}
\newproclaim{rem}[theor]{Remark}
\newproclaim{rems}[theor]{Remarks}
\newtheorem{hyp}{Hypothesis}
\newcommand{\xrightarrow}[1]{\stackrel{#1}{-\!\!\!-\!\!\!\longrightarrow}}
\newcommand{\N}{\mathbb N}
\newcommand{\R}{\mathbb R}
\newcommand{\E}{\mathrm E}
\newcommand{\Var}{\operatorname{Var}}
\newcommand{\p}{\mathrm P}
\makeatother

\begin{document}
\begin{frontmatter}

\title{Resource dependent branching processes\\and the envelope of societies}
\runtitle{The envelope of societies}

\begin{aug}
\author[A]{\fnms{F. Thomas} \snm{Bruss}}
\and
\author[A]{\fnms{Mitia} \snm{Duerinckx}\corref{}\ead[label=e1]{mduerinc@ulb.ac.be}}
\runauthor{F. T. Bruss and M. Duerinckx}
\affiliation{Universit\'e Libre de Bruxelles}
\address[A]{D\'epartement de Math\'ematique\\
Facult\'e des sciences\\
Universit\'e Libre de Bruxelles\\
CP 210, B-1050 Brussels\\
Belgium\\
\printead{e1}} %adresu isvedimo komanda gale!
\end{aug}

% HISTORY:
\received{\smonth{11} \syear{2012}}
\revised{\smonth{10} \syear{2013}}

% ABSTRACT
%
\begin{abstract}
Since its early beginnings, mankind has put to test many different
society forms, and this fact raises a complex of interesting questions.
The objective of this paper is to present a general population model
which takes essential features of any society into account and which
gives interesting answers on the basis of only two natural hypotheses.
One is that societies want to survive, the second, that individuals in
a society would, in general, like to increase their standard of living.
We start by presenting a mathematical model, which may be seen as a
particular type of a controlled branching process. All conditions of
the model are justified and interpreted. After several preliminary
results about societies in general we can show that two society forms
should attract particular attention, both from a qualitative and a
quantitative point of view. These are the so-called weakest-first
society and the strongest-first society. In particular we prove then
that these two societies stand out since they form an envelope of all
possible societies in a sense we will make precise. This result (the
envelopment theorem) is seen as significant because it is paralleled
with precise survival criteria for the enveloping societies. Moreover,
given that one of the ``limiting'' societies can be seen as an extreme
form of communism, and the other one as being close to an extreme
version of capitalism, we conclude that, remarkably, humanity is close
to having already tested the limits.
\end{abstract}

% KEYWORDS

\end{frontmatter}

%s1 #&#
\section{Introduction}
What is the goal of any society? Are there natural boundaries for
societies mankind would not or cannot exceed? And if so, can we
quantify the critical parameters characterizing these boundaries?
Certain aspects of these questions are equally interesting for animal
societies; in fact, throughout this paper we shall always speak of
``individuals'' to make clear that, although the motivation stems from
thinking about man, we keep general populations in mind.

The first question is partially philosophical, and we only treat it in
as much as it concerns the subsequent questions. Here we shall provide
a mathematical answer obtained from a model we propose as a global
mathematical model for societies. This model is built on branching
processes and submitted to two natural hypotheses. Still rudimentary,
the model is broad enough to allow for essential features of life
within any society: reproduction of individuals, the desire to have a
future, heritage and production of resources, consumption of resources,
policies to distribute resources among individuals, and, as a tool of
interaction, the right of emigration. We look at different sub-models
of the model, characterizing different societies. These are defined by
the type of control they exercise through different policies to
distribute resources among their individuals.

%s1.1 #&#
\subsection{Objectives of societies}
Any society is likely to advertise certain keywords in its program or
mission statement, such as justice, liberty, equal opportunity, etc. We
all agree that these issues are likely to be important. However, there
may be as many different interpretations of them as there are
individuals in a population. Hence, within a whole population they can
hardly serve as real guidelines for the choice of a specific society
form. We conclude that any reasonable approach must be more focused.

The philosophy of our approach to answering questions about the choice
of a society is therefore to focus on factors which are seen as
dominant, namely those
which come out of two natural and seemingly inoffensive hypotheses:

\begin{hyp}\label{hyp1}
Individuals want to survive
and to see a future for their descendants.
\end{hyp}

\begin{hyp}\label{hyp2}
Individuals prefer, in general, a higher
standard of living to a lower one.
\end{hyp}

Since these hypotheses may not be compatible with each other,
we define Hypothesis~\ref{hyp1} to have a higher priority than Hypothesis~\ref{hyp2}.

Other hypotheses may be implicit. For instance, the desire to have
security is implicit in Hypothesis~\ref{hyp2}. If the standard of living is
sufficiently high, the society can afford a qualified police force or a
strong army.

To deal with these hypotheses in an adequate way, the problem is to
find a suitable model. This requires two important conditions. First,
the model should allow for all mentioned features which are seen as
essential for the development of a human society and also for a clear
interaction of individuals within the society. Second, it should be
sufficiently tractable to allow for quantifiable conclusions.

%s1.2 #&#
\subsection{History of results}

%F. Thomas Bruss
The first-named author
has been thinking about ways to model societies
for many years. He had given a first talk on \textit{resource dependent
branching processes} in 1983, a second around 1995 and a third in 2001.
Although the publications~\citet{Bruss84} and~\citet{Bruss91} were
motivated by thinking about such processes, this is the very first
paper devoted to this subject.

In the beginning, only preliminary results about necessary conditions
for survival were obtained, and only for an elementary model. These
results were based on earlier work on branching processes with random
absorbing processes, on \mbox{$\varphi$-}branching processes, and on different
forms of the Borel--Cantelli lemma.

In a second step, several models were tested until the model presented
here took its approximate shape. When seeing, in a different context,
the article by \citet{Coffman87}, the results were sharpened to our
needs in~\citet{Bruss91}. These opened the way to quantifiable
conclusions for the chosen model and thus to survival criteria for
several special societies.

In a third step it became visible that, in any reasonable model, two
societies deserved special attention. These are what we call the \textit{strongest-first society} (\mbox{s.f.-}society), and the \textit{weakest-first
society} (w.f.-society). A survival criterion for the w.f.-society was
proved; survival criteria for the s.f.-society were tested, and the idea
of a theorem of envelopment began to emerge.

The fourth step (with the co-author) brought a broad definition of
general policies as well as a proof of a survival criterion for the
s.f.-process. It also led to the precise formulation and proof of \textit{the envelopment theorem for societies. }This theorem says (in both a
conditional and an unconditional form) that all societies are bound to
live in the long run between the s.f.-society and the w.f.-society.
Combined with all earlier findings, we think this is a fundamental result.

%s1.3 #&#
\subsection{Related work}
Our model is an asexual controlled branching process (BP), where \textit{controlled} should be understood in an interacting sense. The general
control is governed by functions of sums of dependent variables, and
self-imposed. This strong dependence property excludes the generating
function machinery, of course. Moreover, although still rudimentary,
the model seems no longer to profit from martingale arguments.

Early work on controlled BPs confined interest to control through
bounds imposed on the growth of Galton--Watson-type processes. \citet{Zubkov74}, \citet{Schuh76} and others modified the number of
individuals which are allowed to reproduce in each generation by
corresponding deterministic functions. \citet{Bruss78} considered a
Galton--Watson process (GWP) with a nonspecified absorbing process for
which only the expected influence is known.

\citet{Yanev76} studied so-called $\phi$-branching processes where the
growth of the GWP reproduction is controlled by random numbers of
offspring which are allowed to reproduce. A more general model for
random control functions was studied in~\citet{Bruss80}, and again in
more generality, by \citet{Gonzales02}. The same authors also examined
$L_2$-convergence for such processes; see \citet{Gonzales05}.

Population-size dependence is another interesting access to control in
BP models. These were studied by \citet{Klebaner85} and \citet{Klebaner86}. \citet{Mannor12} proposed a special class of controlled
BPs involving a different notion of ``resources.'' Motivated by
applications in marketing, the objective is to control independent
subpopulations (multi-type model) in such a way that they grow as
quickly as possible. Relative frequencies of types were studied
in \citet{Yakovlev09}.

The model presented in this paper is neither a BP with varying
environment [see, e.g., \citet{Cohn96}] nor a BP with random
environment. See \citet{Jagers75} for a clear analysis of the connection
between these two types, and, for example,  \citet{Haccou07} for newer
developments. Our model is neither a multi-type BP  nor a pure
population size-dependent model. It is a Markov process, as we shall
see, but no phase-type Markov model or decomposable BP [see \citet{Hautphenne12}] can play the control we have in mind.

Hence, our model does not fit these or similar models studied in the
literature. Nevertheless, related work is sincerely acknowledged. It
has helped, over the years, to get a feeling of what result one can, or
cannot, possibly hope for.

%s2 #&#
\section{The model}
We consider a population, beginning at time $0$ with a fixed number of
individuals, which reproduce at distinct times $n\in\N_0$. The time
interval $[n,n+1)$ is called the $n$th \textit{generation}. Individuals
consume resources and create new resources for their descendants. Only
those descendants whose resource claims will be met by society will
stay within the population until the next reproduction time; the others
are supposed to emigrate (or die) before reproduction. We first define
all the components of the model.

%s2.1 #&#
\subsection{Reproduction}
Individuals are supposed to reproduce independently of each other. The
model supposes that reproduction is asexual. The number of descendants
of each individual is modeled according to a common probability law
$(p_j)_{j\in\N}$, where $p_j$ denotes the probability that a given
individual will have exactly $j$ offspring. To avoid trivial cases, we
suppose $p_0>0$ and $p_j>0$ for at least some $j>1$. Let $D_n^k$ denote
the number of descendants of the $k$th individual in the $n$th
generation. Hence $\p[D_n^k=j]=p_j$, for all $n\in\N,k\in\N_0$ and all
$j\in\N$. The infinite double-array $(D_n^k)_{n\in\N,k\in\N_0}$, named
\textit{reproduction matrix}, thus consists of independent identically
distributed (i.i.d.) integer-valued nonnegative random variables with
mean $m:=\E(D_n^k)<\infty$.

%s2.2 #&#
\subsection{Resources and resource space}
Human beings need food; they need resources. They also reproduce, and
thus they need resources for their descendants. Hence they must save
resources and create resources for future generations.

In our model, individuals inherit resources from preceding generations,
consume resources and create new resources. The resources an individual
can use during his lifetime determines his standard of living. The
society decides in what way resources are distributed among the
individuals, or expressed differently, it is the acceptance of policies
to distribute resources that defines a society. The inherited
resources, plus the newly created ones, are, after deduction of
consumption, considered to be the individual's contribution to the \textit{common} resources of the society, called the resource space.

We do not distinguish between heritage, new production and
nonconsumption of resources and summarize heritage plus production
minus consumption as \textit{creation} of resources. Resource creations of
individuals are modeled as i.i.d. \mbox{real-valued} nonnegative random
variables $R_n^k$, $n\in\N,k\in\N_0$, and the infinite double array
$(R_n^k)_{n\in\N,k\in\N_0}$ will be called \textit{resource creation
matrix}. We suppose that $r:=\E(R_n^k)<\infty$.

%s2.3 #&#
\subsection{Objective of survival}
The population's desire to survive is understood as the objective to
have for the society as a whole a positive probability of surviving
forever. If certain rules to distribute resources allow for a positive
probability of survival, and if other rules do not achieve this, then
the objective to survive takes priority, and the rules are changed
accordingly. It suffices to see changes as an omnipresent option and to
think of the rules defining the society, even if they had been changed
many times before, as being fixed from today onward for the whole
future. (This ``fixed future-instant control'' assumption has the
advantage that society need not be expected to have long-term
prophetical abilities.)

%s2.4 #&#
\subsection{Resource claims within a society}
The model interprets for each descendant, the individual claim of
resources as the outcome of two random components. One is the
descendant's desire to have a certain amount of resources, and the
other is what the descendant, with its own power of conviction, will be
able to defend among its competitors within the society.

These random claims of individuals are modeled as i.i.d. real-valued
nonnegative random variables governed by a known continuous
distribution function $F$. If there are $t_n$ descendants in the $n$th
generation they generate a \textit{string of claims} $(X_n^1, X_n^2,\ldots, X_n^{t_n})$. The\vspace*{1pt} infinite double array $(X_n^k)_{n\in\N,k\in\N_0}$,
is called \textit{claim matrix}. We have $F(x)=\p[X_j^k \le x]$ and
suppose $\mu:=\break \E(X)<\infty$.

%s2.5 #&#
\subsection{Interaction of individuals and society}

Each individual is supposed to have the right to emigrate, and the
control instrument is the right to exercise the option of emigration.
To fix the rules, we suppose that an individual emigrates if and only
if his individual resource claim is not completely satisfied by the
society; otherwise he remains a member of the population until the end
of the generation. Emigration is supposed to happen before an
individual produces offspring. Hence each individual resource
assignment (seen as the individual standard of living offered by the
society) is felt by an individual as being either sufficient, implying
``stay,'' or else insufficient, implying ``leave.''

Typically, the total resource space created by a generation is
insufficient to satisfy all the resource claims of the offspring. We
define a \textit{policy} as a function which determines then a priority
order among offspring, that is, a rule to distribute the resources
created by the current generation among the next generation.

%s2.5.1 #&#
\subsubsection{Examples}
To keep examples simple we use here positive integers for claims and
available resources; this is not required in reality, of course. Any
random claim expresses the number of units of resources the individual requires.
Assume, for instance, that, in a given generation, the number of
individuals is $10$, and that the current resource space is $100$.
Suppose further that the individuals write down their claims on a list
in some order, as, for instance, in chronological order of arrival of
claims, and that the string of claims reads
\[
11, 7, 15, 19, 11, 18, 10, 22, 17, 19.
\]
The ``first-come-first-served'' society would grant the first seven
claims\break (adding up to $91$), and the last three applicants would then
have to emigrate. The $9$ remaining units may be divided among the
seven or put back into the common resources. (Details on this level
will not matter for our results.)
A society that distributes resources in f.c.f.s.-order may not have much
appeal, as one may argue. However there are certainly more foolish
policies, as, for example, the ``coin-flipping policy'' which chooses
the priority order at random. Any procedure to select a priority of
claims is considered a policy.

In the sequel, two particular policies will attract our special
interest: the first one, called the weakest-first society, satisfies
the smallest claims first and would thus retain, in the example above,
the claims $7,10,11,11,15,17,18$, while the other, called the
strongest-first society, satisfies the largest claims first and would
thus retain the claims $22,19,19,18,17$.

%s2.6 #&#
\subsection{Resource dependent branching processes}
We now give a precise definition of the type of population processes we
consider in this paper. Two definitions are needed. Let $((D^k_n)_{n\in
\N,k\in\N_0},(X^k_n)_{n\in\N,k\in\N_0},(R^k_n)_{n\in\N,k\in\N
_0})$ be a
triplet of independent double arrays of i.i.d. random variables defined
on a probability space $(\Omega,\mathcal F,\p)$. As before, the
variables $D_n^k$, $X_n^k$ and $R_n^k$ ($k\in\N_0$) represent the
number of offspring, the resource claims and the production of
resources (resp.) of each individual (labeled by $k$) in generation
$n$. We always assume that these variables satisfy the natural
regularity conditions given below; see Section~\ref{chapnatregcond}.
Let
%
%e1 #&#
%
\begin{equation}
\label{sums} D_n(k):=\sum_{j=1}^kD_n^j
\quad\mbox{and}\quad R_n(k):=\sum_{j=1}^kR_n^j
\end{equation}
denote the total number of offspring and the total resources created by
generation~$n$, respectively, given that generation $n$ counts $k$
individuals. The i.i.d. assumptions for random variables within the
same double array allow us to use the shorter notation $D(k)=D_n(k)$
and $R(k)=R_n(k)$ whenever we limit our interest to their \textit{distributional} prescriptions. Conversely, this is understood
throughout the paper whenever we use this simplified notation.

We first need a precise definition of a policy:

%de2.1 #&#
%
\begin{defin}[(Global definition of a policy)]\label{defpol}
A \textit{policy} is a sequence $\pi=(\pi_t)_{t\in\N}$, where, for all
$t\in\N$, $\pi_t$ is a function associating to any $t$-uple
$(x_k)_{k=1}^t\in(\R^+)^t$ a permutation $\pi_t((x_k)_{k=1}^t)\in
\operatorname{Sym}(t)$ of the set $[t]:=\{1,\ldots,t\}$.
\end{defin}

In this definition, $t$ corresponds to the number of offspring, and
$(x_k)_{k=1}^t$ to their respective resource claims. The permutation
$\pi_t((x_k)_{k=1}^t)\in\operatorname{Sym}(t)$ then gives the priority order
that the society has chosen to satisfy the claims of the offspring: the
individual $\pi_t((x_k)_{k=1}^t)(1)$ is the first served, etc. If $s$
denotes the total of resources produced by the previous generation, the
number of offspring having their claims completely satisfied thanks to
the society's policy $\pi$ is thus defined by
\[
Q^\pi \bigl(t,(x_k)_{k=1}^t,s
\bigr)= \cases{0,\qquad\mbox{if $t=0$ or $x_{\pi_t((x_k)_{k=1}^t)(1)}>s$},
\cr
\displaystyle\max \Biggl\{1\le k\le t\dvtx \sum_{j=1}^kx_{\pi
_t((x_k)_{k=1}^t)(j)}
\le s \Biggr\},
\cr
\hspace*{32.5pt} \mbox{otherwise.}}
\]
Note that this function $Q^\pi$ necessarily satisfies
\[
Q^\pi (0,\varnothing,s )=0=Q^\pi \bigl(t,(x_k)_{k=1}^t,0
\bigr)\quad\mbox{and}\quad0\le Q^\pi \bigl(t,(x_k)_{k=1}^t,s
\bigr)\le t,
\]
for all $s\in\mathbb R^+$, all $t\in\mathbb N$ and all
$(x_k)_{k=1}^t\in
(\R^+)^t$. Recall that all the offsprings that are not completely
satisfied, and only these, leave the society forever. This leads to the
definition of the following stochastic process:

%de2.2 #&#
%
\begin{defin}[(Global model)]
If $\pi$ is a policy, the \emph{resource dependent branching process}
(RDBP) \emph{on $(D_n^k,X_n^k,R_n^k)_{n,k}$ controlled by $\pi$} is
defined as the integer-valued, nonnegative stochastic process $(\Gamma
_n)_{n\in\N}$, defined by $\Gamma_0=1$ and recursively
\[
\Gamma_{n+1}=Q^\pi \bigl(D_n(
\Gamma_n),\bigl(X_n^k\bigr)_{k=1}^{D_n(\Gamma
_n)},R_n(
\Gamma_n) \bigr),
\]
where $D_n(\cdot)$ and $R_n(\cdot)$ are given by equation~(\ref{sums}).
\end{defin}

%s2.6.1 #&#
\subsubsection{Remarks}
(i) The notation $(\Gamma_n)_n$ is mnemonic for ``general'' in the sense
that the policy $\pi$ in $Q^\pi$ is not specified, and this is
maintained throughout this paper.

(ii) Unless specified otherwise, each process in
this paper is supposed to start at time $0$ at level $1$; exceptions to
this will be clearly indicated.

(iii) Concerning all independence assumptions, we
realize, of course, that in a convincing model, the random variables
$D_n^k$, $R_n^k$ and $X_n^k$ should allow for some interaction
(dependence), and the i.i.d. assumption is primarily made for
simplicity. However, it is important to note that,
in our setting, this assumption is less restrictive than it may seem.
Indeed, recall that Hypothesis~\ref{hyp1} is given priority to Hypothesis~\ref{hyp2}. If
a population wants to know whether survival is possible, it must look
at the current situation, because we do not assume in the model that
the population knows more about the long-term future. Therefore the
question is what would happen if the current situation were maintained
for the future. Each time a change is warranted, for instance, an
encouragement to have more descendants, or to increase resource
creation, the matrices can be exchanged. It is this \textit{instant
control} mentioned earlier which gives considerable support to all
independence assumptions.

%s2.7 #&#
\subsection{Regularity assumptions}\label{chapnatregcond}
We suppose that the following assumptions are always satisfied:

\begin{longlist}[(iii)]
\item$1<m<\infty$, $r<\infty$ and $0<\mu<\infty$;
\item$p_0>0$ and there exists some $k\ge2$ with $p_k>0$;

\item the trio of laws of reproduction, creation of resources and
claims is compatible with a positive probability, however small it
might be, that the process can reach any finite state;\vspace*{1pt}
\item the variables $(D_n^k)_{n,k}$, $(R_n^k)_{n,k}$ and
$(X_n^k)_{n,k}$ all have finite variance;
\item(the random variables $D_n^k$, $R_n^k$ and $X_n^k$ are all bounded).
\end{longlist}

%s2.7.1 #&#
\subsubsection{Justification of assumptions}\label{chapjustif}
In assumption~(i), the conditions $m>1$ and $\mu>0$ do not restrict
generality: the case $m\le1$ is trivial because then \textit{any} RDBP is
stochastically smaller than a subcritical GWP. With the natural
condition $p_0>0$ of (ii) it is bound to die out. The case $\mu=0$
implies that $(X_n^k)_{n,k}$ consists only of $0$'s, so that the
process coincides with the standard GWP. Survival is thus possible if
and only if $m>1$, implying $p_k>0$ for some $k \ge2$,
hence (ii).

Assumption (iii) ensures that the process can grow. It is, for instance,
satisfied if we assume that $F(r/k)>0$ for some $k\ge2$ with $p_k>0$.
Note that this assumption becomes superfluous if we replace the initial
setting $\Gamma_0=1$ by $\Gamma_0=L$ for some $L$ sufficiently large.

The assumption of finite variances of all random variables is needed
for our results and is also completely realistic.

Finally, assumption (v) of boundedness is, apart for two results (i.e.,
Theorems~\ref{th2} and~\ref{thbounds}), not needed and therefore put
in brackets. Note that even this stronger assumption is well defendable
in our model, at least for human societies.

%s2.8 #&#
\subsection{Multi-parameter policies}
According to our definition, a policy can only depend on the available
resources and on the claims of the offspring. However, in more
realistic models, the offspring could be characterized by many other
different parameters, and it would be natural to allow a policy to
depend on all these additional parameters. This is why, although we do
not pursue such general models in this paper, we will indicate shortly
how to adapt our definitions accordingly.

We consider a new double array $({\mathbf Y}_n^k)_{n,k}$ of i.i.d. random
$p$-vectors defined on a corresponding probability space $(\Omega,\mathcal F,P)$. Here, the components of the random vectors ${\mathbf
Y}_n^k$ ($k\in\N_0$) correspond to the different characteristic
parameters of each individual (labeled by $k$) in generation $n$. For
some fixed $p\ge0$, a \textit{$p$-parameter policy} is any sequence
$\pi
=(\pi_t)_{t\in\N}$, where, for all $t\in\N$, $\pi_t$ is a function
associating to any $t$-uple $(x_k,{\mathbf y}_k)_{k=1}^t\in(\R^+\times
B)^t$ a permutation $\pi_t((x_k,{\mathbf y}_k)_{k=1}^t)\in\operatorname{Sym}(t)$
of the set $[t]:=\{1,\ldots,t\}$, where $B\subset\R^{p}$ denotes the
set of possible parameter values. The associated counting function and
the associated RDBP are defined as before. For instance, the
coin-flipping policy could be seen as a trivial example of a
multi-parameter policy, where the coin-flipping parameter actually
determines the whole policy.

Note that the situation is trivial when the additional parameters of an
individual are assumed to be independent of its number of offspring,
its resource claim and its resource production, and when we consider
some multi-parameter policy that only depends on these additional
parameters (but not on the resource claims): in this case, the
associated RDBP has exactly the same behavior as the f.c.f.s.-process (as
defined below). In general, the dependence may of course lead to highly
complex situations.

%s3 #&#
\section{Particular policies}\label{expol}
In the following, we define policies of particular interest. The first
will be a neutral policy, which we call the \textit{first-come-first-served} policy. It will serve as a point of comparison
with the \textit{weakest-first} policy and the \textit{strongest-first}
policy defined later.

%s3.1 #&#
\subsection{First-come-first-served policy}
The f.c.f.s.-policy is a neutral policy in the sense that it serves the
claims according to their respective arrival times. To exclude
ambiguities in the definition, these arrivals of claims are supposed to
happen at the beginning of each generation, being almost surely
different, and all preceding the times of producing offspring.

%de3.1 #&#
%
\begin{defin}
The \textit{first-come-first-served policy} (f.c.f.s.-policy) is the policy
$\pi^U$ defined by $\pi_t^U((x_k)_{k=1}^t)=\mathrm{id}_{[t]}$.\footnote
{The notation $\pi^U$ should remind of the unordered $x_t^1,\ldots,x_t^t$ used in the definition.}
\end{defin}

The associated function $C:=Q^{\pi^U}$ counting the
individuals staying in the process is
\begin{eqnarray*}
C \bigl(t,(x_k)_{k=1}^t,s \bigr)= %
\cases{ 0, & \quad if $t=0$ or $x_1>s$;
\cr
\displaystyle
\operatorname{sup} \Biggl\{1\le k\le t\dvtx  \sum_{j=1}^k
x_j\le s \Biggr\}, &\quad otherwise.} %
\end{eqnarray*}

%de3.2 #&#
%
\begin{defin}
The \emph{first-come-first-served process} (f.c.f.s.-process) on
$(X_n^k,D_n^k,R_n^k)_{n,k}$ is the RDBP controlled by $\pi^U$, that is,
the stochastic process $(U_n)_{n\in\mathbb N}$ defined by $U_0=1$, and
recursively by
\[
U_{n+1}=C \bigl(D_n(U_n),
\bigl(X_n^k\bigr)_{k=1}^{D_n(U_n)},R_n(U_n)
\bigr).
\]
\end{defin}

Note that $C(t,(X^k_n)_{k=1}^t,s)+1$ is a stopping time with
respect to the natural filtration $({\mathcal F}_\ell)_\ell$, where
${\mathcal
F}_\ell$ denotes the $\sigma$-field generated by the $X_n^k$'s for
$1\le k\le\ell$. It is useful to refer to stopping-time properties
because frequently we use results which become intuitive if we think of
a version of ``Wald's lemma'' for curtailed random variables; see
Section~4 of \citet{Bruss91}.

\subsubsection*{Interpretation and properties} The f.c.f.s.-society may be seen as a
model of a \textit{laissez-faire} society. When individuals are born, they
are assumed to arrive at different times within their generation at
maturity and then submit their random resource claims. This continues
as long as resources are available. Since the claims are i.i.d. random
variables, it is not the society but the scarcity of resources that
imposes constraints. This process has some similarity with the GWP
because, for given distributions of resource creation and claims, the
claims curtail the effective mean $m$ of the offspring distribution
$(p_k)_k$. However, given that the process depends in each generation
on common resources, the similarity with a GWP is still rather limited.

%s3.2 #&#
\subsection{Weakest-first policy}
The weakest-first policy (w.f.-policy) is an extreme policy, giving
priority successively to the least demanding currently remaining offspring.

%de3.3 #&#
%
\begin{defin}\label{defNpol}
The \textit{weakest-first policy} (w.f.-policy) is the policy $\pi^W$
defined by $\pi_t^W((x_k)_{k=1}^t)=\sigma$, where $\sigma$ is the
permutation of $[t]$ such that $x_{\sigma(1)}\le\cdots\le x_{\sigma(t)}$.
\end{defin}

Throughout this paper, for i.i.d. realizations
$(x_k)_{k=1}^t$ of the random variable $X$, the increasing order
statistics will be denoted by $x_{1,t}\le x_{2,t}\le\cdots\le x_{t,t}$.
The associated counting function $N:=Q^{\pi^W}$ is now
%
%e2 #&#
%
\begin{equation}
\label{defN} N \bigl(t,(x_k)_{k=1}^t,s \bigr)=
\cases{ 0, & \quad if $t=0$ or $x_{1,t}>s$,
\cr
\displaystyle\sup \Biggl\{1\le k\le t\dvtx  \sum_{j=1}^k
x_{j,t}\le s \Biggr\}, & \quad otherwise.} \hspace*{-25pt}
\end{equation}

%de3.4 #&#
%
\begin{defin}\label{defw.f.}
The \emph{weakest-first process} (w.f.-process) on
$(D_n^k,X_n^k,\break R_n^k)_{n,k}$ is the RDBP controlled by $\pi^W$, that is,
the stochastic process $(W_n)_{n\in\mathbb N}$ defined by $W_0=1$, and
recursively by
%
%e3 #&#
%
\begin{equation}
W_{n+1}=N \bigl(D_n(W_n),
\bigl(X_n^k\bigr)_{k=1}^{D_n(W_n)},R_n(W_n)
\bigr).
\end{equation}
\end{defin}

Note that $N(\cdot,\cdot,\cdot)$ counts the maximal number of
increasing order statistics of the random sample $(x_k)_{k=1}^t$ which,
starting with the smallest, can be summed up without exceeding $s$.
Further, $N(t,(X_n^k)_{k=1}^t,s)+1$ is a stopping time on the
filtration $({\mathcal F}_\ell^I)_\ell$ say, generated by the $\ell$ first
increasing order statistics from all order statistics, beginning with
the smallest one, but it is not a stopping time with respect to the
natural filtration $({\mathcal F}_\ell)_\ell$.

\subsection*{Interpretation and properties} The policy of the w.f.-society is to
support always the weakest. In that respect it comes close to the ideas
of socialism and communism. In each generation, individuals are ordered
according to their resource claims, and these order statistics are
highly dependent of each other.

The following lemma will be needed throughout.

%
%le3.5 #&#
%
\begin{Lem}\label{lemN}
$N(t,(x_k)_{k=1}^t,s)$ is increasing in both $t$ and $s$.
\end{Lem}

\begin{pf}
This follows immediately from Definition~\ref{defNpol}.
\end{pf}

%s3.3 #&#
\subsection{Strongest-first policy}
The strongest-first policy (s.f.-policy) gives successively priority to
the most demanding currently remaining offspring, that is to the
largest random claims.

%de3.6 #&#
%
\begin{defin}
The \textit{strongest-first policy} (s.f.-policy) is the policy $\pi^S$
defined by $\pi_t^S((x_k)_{k=1}^t)=\sigma$, where $\sigma$ is the
permutation of $[t]$ such that $x_{\sigma(1)}\ge\cdots\ge x_{\sigma(t)}$.
\end{defin}

The associated counting function $M:=Q^{\pi^S}$ becomes
%
%e4 #&#
%
\begin{equation}
\label{defM} M \bigl(t,(x_k)_{k=1}^t,s \bigr)=
\cases{ 0, \qquad\mbox{if $t=0$ or $x_{t,t}>s$,}
\cr
\displaystyle
\sup \Biggl\{1\le k\le t\dvtx  \sum_{j=t-k+1}^tx_{j,t}
\le s \Biggr\},
\cr
\hspace*{32.5pt} \mbox{otherwise.} }
\end{equation}
It counts the maximal number of decreasing order statistics which can
be summed up, starting with the biggest, without exceeding $s$.

%de3.7 #&#
%
\begin{defin}
The \emph{strongest-first process} (s.f.-process) on
$(X_n^k,D_n^k,\break R_n^k)_{n,k}$ is the RDBP controlled by $\pi^S$, that is,
the stochastic process $(S_n)_{n\in\mathbb N}$ defined by $S_0=1$, and
recursively by
%
%e5 #&#
%
\begin{equation}
S_{n+1}=M \bigl(D_n(S_n),
\bigl(X_n^k\bigr)_{k=1}^{D_n(S_n)},R_n(S_n)
\bigr).
\end{equation}
\end{defin}

We note that $M(t,(X_n^k)_{k=1}^t,s)+1$ is a stopping time on
the filtration $({\mathcal F}_\ell^D)_\ell$ generated by the first
$\ell$
decreasing order statistics of all currently presented claims,
beginning with the largest one. It is again no stopping time on the
natural filtration~$({\mathcal F}_\ell)_\ell$.

\subsubsection*{Interpretation and properties} The s.f.-society is the model which
serves the strongest
individuals first. Since we identified the values of resource claims
with the power to defend these claims, this society shares important
features with free-market policies and an uncontrolled capitalistic society.
Since claims are again highly dependent, the technical difficulty in
this model is comparable with the one evoked for the w.f.-society.

For a closer study of the s.f.-process, we will need later the
following definition:

%de3.8 #&#
%
\begin{defin}
We say that a function $h\dvtx [t_1,t_2]\to\R$ defined on an interval
$[t_1,t_2]\subset\R$ is \textit{cap-unimodal} if it is either monotone, or
else unimodal and cap-shaped, on $[t_1,t_2]$.
\end{defin}

Note that a cap-unimodal function $h$ on $[t_1,t_2]$ satisfies
%
%e6 #&#
%
\begin{equation}
\min_{t\in[t_1,t_2]}h(t)=\min\bigl\{h(t_1),h(t_2)
\bigr\},
\end{equation}
provided that $h$ is defined in both $t_1$ an $t_2$. The following
lemma then contrasts Lemma~\ref{lemN}:

%le3.9 #&#
%
\begin{Lem}\label{lemunimod}
$M(t,(x_k)_{k=1}^t,s)$ is increasing in $s$ for fixed $t$, and, for
fixed $s$, cap-unimodal in $t$ on any interval $[t_1,t_2]$. Further,
$\max_{t\in[t_1,t_2]}M(t,(x_k)_{k=1}^{t},\break s)\le
t_2-t_1+M(t_1,(x_k)_{k=1}^{t_1},s)$.
\end{Lem}

\begin{pf}
See Section~\ref{prmarkov}.
\end{pf}

%re3.10 #&#
%
\begin{rem}
If resources are plenty and suffice to accommodate all claims, then all
policies have the same effect; that is, they allow all individuals to
stay and reproduce. However, if not, the w.f.-society is the one which
allows the maximum number of individuals to stay and to reproduce. The
s.f.-society is then opposite in the sense that the resource space is
used up by the corresponding minimum number of applicants.
\end{rem}

%s4 #&#
\section{Main results}\label{results}
Throughout this section, all RDBPs are supposed to be controlled by
some policy $\pi$ on $(D_n^k,X_n^k,R_n^k)_{n,k}$, where all random
variables satisfy the assumptions of Section~\ref{chapnatregcond}.
%s4.1 #&#
\subsection{Preliminaries}
It is important to first point out that any RDBP shares the following
property, which is typical for many branching processes. Namely, either
it explodes, or it becomes extinct.

%pr4.1 #&#
%
\begin{Prop}[(Markov property)]\label{markov}
Any RDBP $(\Gamma_n)_n$ is a Markov process with a unique absorbing
state, which is $0$. Moreover, it tends a.s. either to $0$ or to~$\infty$.
\end{Prop}

\begin{pf}
See Section~\ref{prmarkov}. (The same result remains true
in the multi-parameter case.)
\end{pf}

In accordance with Hypothesis~\ref{hyp1}, we must first answer the question
under which conditions a given RDBP $(\Gamma_n)_n$ can survive, that
is, we must determine when the extinction probability
\[
q_\Gamma=\p \Bigl[ \lim_{n\to\infty}
\Gamma_n=0 \big| \Gamma _0=1 \Bigr]
\]
is equal to $1$.
Note that, in the case $q_\Gamma<1$, the probability of extinction
could intuitively be made arbitrarily small if we replace the initial
setting $\Gamma_0=1$ by $\Gamma_0=M$ for $M$ sufficiently large. We
will in fact prove this for several processes, and for the w.f.-process
this holds even in a stronger form:

%pr4.2 #&#
%
\begin{Prop}[(``Safe-haven'' property of the w.f.-process)]\label{propboundM}
For all $L\in\N_0$,
\[
\p \Bigl[ \lim_{n\to\infty}W_n=0 \big|
W_0=L \Bigr]\le q_W^L.
\]
\end{Prop}

\begin{pf}See Section~\ref{prmarkov}.
\end{pf}

Hence, if a society fears extinction it may change to become
a w.f.-society and likely survive unless $q_W=1$, or $L$ is small. Also,
as we shall see later on, $q_W=1$ implies $q_\Gamma=1$ for any RDBP
$(\Gamma_n)_n$, so that in that case no change in policy could avoid
extinction. The w.f.-society may be seen as the ``safe-haven'' society
form with respect to Hypothesis~\ref{hyp1}.

%s4.2 #&#
\subsection{Uniform upper-bound process}\label{sbounds}
It turns out that the w.f.-process is always an upper bound for any other
RDBP, and this in the strongest sense:

%pr4.3 #&#
%
\begin{Prop}[(Uniform upper bound)]\label{propbounds}
Let $(\Gamma_n)_n$ be any RDBP, and let $(W_n)_n$ be the w.f.-process
defined on the same double arrays. Then, for all $n$, we have $\Gamma
_n\le W_n$ a.s. In particular, $q_W\le q_\Gamma$.
\end{Prop}

\begin{pf}See Section~\ref{prbounds}. (The same result remains true
in the multi-parameter case.)
\end{pf}

%s4.2.1 #&#
\subsubsection{Nonexistence of a uniform lower-bound process}\label{chapnolow}
We now turn to a comparison between $(\Gamma_n)_n$ and the
corresponding s.f.-process $(S_n)_n$. This is a more subtle problem.
Indeed, it is in general not true that $S_n\le\Gamma_n$ a.s. for all $n$.

This may come somewhat as a surprise.
Indeed, since the s.f.-society is clearly the most restrictive one for
the number of offspring which can stay, one feels that $(\Gamma_n)_n$
should always do at least as well as the process $(S_n)_n$ governed by
the s.f.-policy.
An explicit counterexample is given in Section~\ref{prbounds}: it is
based on the fact that $M(t,(x_k)_{k=1}^t, s)$ is, for fixed $s$,
increasing in $t$ up to some threshold $t_s$ but decreasing for $t\ge
t_s$. However we can explain here already what is behind it.

Suppose $(S_n)_n$ and $(\Gamma_n)_n$ have the same number $k$ of
individuals at time $n$. Then it follows from the counting function
comparison that $\Gamma_{n+1}$ is at least as large as $S_{n+1}$. Hence
we expect on average more offspring from $\Gamma_{n+1}$ than from
$S_{n+1}$. But then the extreme claims of the offspring of $\Gamma
_{n+1}$ must be expected to be larger than those from the offspring of
$S_{n+1}$. If the policy of $(\Gamma_n)_n $ serves just one of the
larger claims, the inequality established in generation $n+1$ may point
to the opposite direction in generation $n+2$.

Therefore, we see that no nontrivial uniform lower bound can exist for
general RDBPs, and thus all attempts to compare general trajectories
would be fruitless. We found it highly interesting that, nevertheless,
we can prove the Envelopment theorem presented in Section~\ref
{chapenvelop} (see Theorem~\ref{thbounds}). This will justify the
fact that we can essentially restrict our attention to the w.f.-policy
and the s.f.-policy, which we will study in the next sections. The
f.c.f.s.-policy will be considered as a point of comparison later on (see
Section~\ref{chapf.c.f.s.}).

%s4.3 #&#
\subsection{Extinction criterion for the w.f.-process}\label{chapw.f.}
%
%th4.4 #&#
%
\begin{theor}\label{th1}
Let $(W_n)_n$ be the w.f.-process on $(D_n^k,X_n^k,R_n^k)_{n,k}$.

\begin{longlist}[(a)]
\item[(a)] If $r\le m\mu$ and if $\tau$ is the solution of
%
%e7 #&#
%
\begin{equation}
\label{tauequ} \int_0^\tau x \,\mathrm{d}F(x)=
\frac{r}m,
\end{equation}
then:

\begin{enumerate}[(ii)]
\item[(i)] if $mF(\tau)<1$, then $q_W=1$;
\item[(ii)] if $mF(\tau)>1$, then $q_W<1$.
\end{enumerate}

\item[(b)] If $r>m\mu$, then $q_W<1$.
\end{longlist}
Moreover, in cases \textup{(a)(ii)} and \textup{(b)}, we even have
%
%e8 #&#
%
\begin{equation}
\label{eqsupbornew} \p \Bigl[ \lim_{n\to\infty}W_n=0
\big| W_0=L \Bigr]\xrightarrow{L\to\infty}0.
\end{equation}
Further, if there is no extinction, the process explodes a.s. and
behaves more and more like a supercritical GWP with a new reproduction
mean $\tilde m(>1)$, say, defined by
\[
\tilde m= %
\cases{ m, &\quad if $r\ge m\mu$,
\cr
m F(\tau),&\quad if
$r<m\mu$ and $m F(\tau)>1$.} %
\]
\end{theor}

\begin{pf}See Section~\ref{prth1}. Equation~(\ref{eqsupbornew})
follows from Proposition~\ref{propboundM}.
\end{pf}

The following remarks will provide a better understanding of these results.

%re4.5 #&#
%
\begin{rems}
(i) The case {(b)} is the most intuitive one. Indeed, the condition
$r>m\mu$ means that a typical ancestor creates in expectation more
resources than his offspring will claim together. Consequently, when
the population grows the law of large numbers ensures that the process
will behave more and more like a supercritical GWP, the asymptotic
properties of which are well understood [see, e.g., \citet{Bingham74}].
For this argument to hold, the regularity assumption~(iii) (see
Section~\ref{chapnatregcond}) is needed to ensure that the process can
reach any finite size with positive probability; this condition becomes
redundant if we replace the initial setting $W_0=1$ by $W_0=w$ for $w$
sufficiently large.

(ii) Theorem~\ref{th1} is sharp in the sense that $mF(\tau)=1$ is the
exact separation point between a.s. extinction and positive survival
probability. However, unlike what occurs with GWPs, it is here not
immediate to see under which conditions on the law $(p_k)_k$ and on $F$
the critical case implies a.s. extinction. Note that, for fixed $m$ and
$F$, the parameter $\tau=\tau(r/m)$ is increasing in $r$, so that the
equation $mF(\tau)=1$ defines a critical mean resource production
$r_{W,c}$ below which $q_W=1$ and above which $q_W<1$.
\end{rems}

Note that the survival conditions depend deeply on the
distribution $F$ of the claims. The following special cases give
criteria in terms of the first two moments only. From the point of view
of applications, this is more attractive since $F$ may not be known precisely.

%co4.6 #&#
%
\begin{cor}\label{cor1}
Let $(W_n)_n$ be the w.f.-process on $(D_n^k,X_n^k,R_n^k)_{n,k}$.

\begin{longlist}[(ii)]
\item If $\mu<r$, we have $q_W<1$.
\item Assume $r\le m\mu(1-\sqrt{1-1/m})$. If $\Var X<\frac{(m\mu
-r)^2}{m(m-1)}-\mu^2(>0)$, we have $q_W=1$.
\end{longlist}
\end{cor}

\begin{pf}See Section~\ref{prcors}.
\end{pf}

%s4.4 #&#
\subsection{Extinction criterion for the s.f.-process}\label{chapsf}
We now present the extinction criterion for the s.f.-process. Since we
deal here again with a process depending on the partial sum behavior of
order statistics---now on the sum of the largest ones---we expect
analogies. To facilitate a comparison between the s.f.-process and the
w.f.-process we had made the assumption [recall~(v) in Section~\ref
{chapnatregcond}] that resource claims are bounded above.

However, many important difficulties will arise, and the comparison
with the w.f.-process will only be possible for a very small part of the
proof. In particular, we will need here the boundedness of all the
random variables $D_n^k$, $X_n^k$ and $R_n^k$ [see~(v) in Section~\ref
{chapnatregcond}].

%th4.7 #&#
%
\begin{theor}\label{th2}
Let $(S_n)_n$ be the s.f.-process on $(D_n^k,X_n^k,R_n^k)_{n,k}$.

\begin{longlist}[(a)]
\item[(a)] If $r\le m\mu$ and if $\theta$ is the solution of
%
%e9 #&#
%
\begin{equation}
\label{thetaequ} \int_\theta^bx \,\mathrm{d}F(x)=
\frac{r}m,
\end{equation}
then:

\begin{enumerate}[(ii)]
\item[(i)] if $m(1-F(\theta))<1$, then $q_S=1$;
\item[(ii)] if $m(1-F(\theta))>1$, then $q_S<1$.
\end{enumerate}
\item[(b)] If $r>m\mu$, then $q_S<1$.\vadjust{\goodbreak}
\end{longlist}
Moreover, in cases \textup{(a)(ii)} and \textup{(b)}, we even have
%
%e10 #&#
%
\begin{equation}
\label{eqsupbornes} \p \Bigl[ \lim_{n\to\infty}S_n=0
\big| S_0=L \Bigr]\xrightarrow {L\to\infty}0.
\end{equation}
Further, if there is no extinction, the process explodes a.s. and
behaves more and more like a supercritical GWP with a new reproduction
mean $\tilde m(>1)$, say, defined by
\[
\tilde m=\cases{ m, &\quad if $r\ge m\mu$,
\cr
m \bigl(1-F(\theta)\bigr), &\quad
if $r<m\mu$ and $m \bigl(1-F(\theta)\bigr)>1$.} %
\]
\end{theor}

\begin{pf}See Sections~\ref{proofth2} and~\ref{proofth2+}.
\end{pf}

%
%re4.8 #&#
%
\begin{rem}
Equation~(\ref{eqsupbornes}) rejoins Proposition~\ref{propboundM} in
a weaker sense. The critical case is now determined by the equation
$m(1-F(\theta))=1$. Note that, for fixed $m$, the parameter $\theta
=\theta(r/m)$ is decreasing in $r$, in the same way that, in
Theorem~(\ref{th1}), $\tau(r/m)$ was increasing in $r$. The equation
$m(1-F(\theta))=1$ thus defines the critical mean resource production
$r_{S,c}$.
\end{rem}

Note that the survival conditions depend deeply on the distribution of
the claims and are thus quite difficult to interpret in practice. The
following special cases, expressed in terms of the two first moments of
the claims only, are more easy to interpret:

%co4.9 #&#
%
\begin{cor}\label{cor2}
Let $(S_n)_n$ be the s.f.-process on $(D_n^k,X_n^k,R_n^k)_{n,k}$.

\begin{longlist}[(ii)]
\item If $r<\mu$, we have $q_S=1$.
\item Assume $r\ge\mu\sqrt{m}$. If $\Var X<{r^2}/m-\mu^2(>0)$, then we
have $q_S<1$.
\end{longlist}
\end{cor}

\begin{pf}See Section~\ref{prcors}.
\end{pf}

%s4.5 #&#
\subsection{Extinction criterion for the f.c.f.s.-process}\label{chapf.c.f.s.}
As a term of comparison, it is interesting to observe what happens in
the case of a f.c.f.s.-process.

%pr4.10 #&#
%
\begin{Prop}\label{prop3}
Let $(U_n)_n$ be the f.c.f.s.-process on $(D_n^k,X_n^k,R_n^k)_{n,k}$.

\begin{longlist}[(a)]
\item[(a)] If $r<\mu$, then $q_U=1$.
\item[(b)] If $r>\mu$, then $q_U<1$.
\end{longlist}
Moreover, in case \textup{(b)}, we even have
%
%e11 #&#
%
\begin{equation}
\label{eqsupborneu} \p \Bigl[ \lim_{n\to\infty}U_n=0
\big| U_0=L \Bigr]\xrightarrow {L\to\infty}0.
\end{equation}
Further, if there is no extinction, the process explodes a.s. and
behaves more and more like a supercritical GWP with reproduction mean $m$.
\end{Prop}

%re4.11 #&#
%
\begin{rem} As in the case of the w.f.-process, the regularity assumption~(v) (see Section~\ref{chapnatregcond}) is not needed in the proof of
the above result. The critical mean resource production is now simply
defined by $r_{U,c}=\mu$.
\end{rem}

%s4.6 #&#
\subsection{Envelopment theorems}\label{chapenvelop}
As explained in Section~\ref{sbounds}, although the \mbox{w.f.-}process
constitutes a uniform upper bound process, no nontrivial uniform lower
bound process can possibly exist for general RDBPs. In this section we
shall see that, however, the s.f.-process constitutes a lower bound
process in a sense that is strong enough to call it an envelopment from
below. Firstly, conditioned on survival, $(S_n)_n$ has the lowest
limiting growth rate of all RDBPs. Secondly, if an arbitrary RDBP
$(\Gamma_n)_n$ cannot survive, the s.f.-process $(S_n)_n$ cannot
survive either.

%s4.6.1 #&#
\subsubsection{Conditional envelopment theorem}
Let us first consider a general RDBP $(\Gamma_n)_n$. Since $0$ is an
absorbing state, we define $\Gamma_{n+1}/\Gamma_n=0$ if $\Gamma_n=0$.
If $\Gamma_{n+1}>0$ we may see $\Gamma_{n+1}/\Gamma_n$ as the empirical
growth rate in period $n$. We know that for some societies the
empirical growth rates will converge a.s. to a limit in time, as, for
instance, for the w.f.-process, the s.f.-process, the f.c.f.s.-process, and
others. But then, given our very general definition of a policy $\pi$,
it is also clear that there are many processes for which the empirical
growth rates do not converge; it suffices to think, for example, of
societies which apply very different rules according to the number of
claims being even or odd.

The following result shows that, conditioned on survival, the growth
rates of any RDBP will finally be between the growth rates of the
w.f.-process and the s.f.-process.

%pr4.12 #&#
%
\begin{Prop}\label{propenvelop1}
Let $(\Gamma_n)_n$ be any RDBP on $(D_n^k,X_n^k,R_n^k)_{n,k}$. Let
\[
\underline\gamma=\liminf_{n\to\infty}\frac{\Gamma_{n+1}}{\Gamma
_n},\qquad
\bar\gamma=\limsup_{n\to\infty}\frac{\Gamma
_{n+1}}{\Gamma_n}.
\]
Given that $\Gamma_{n}\to\infty$ (i.e., $\underline\gamma>0$), we have
\[
m\bigl(1-F(\theta)\bigr)\le\underline\gamma\le\bar\gamma\le mF(\tau),
\]
where $\tau$ and $\theta$ are defined as in Theorems~\ref{th1}
and~\ref{th2}.
\end{Prop}

\begin{pf}
See Section~\ref{prenvelop1}.
\end{pf}

Hence, there may be no limiting growth rate of a RDBP, but the $\liminf
$ and the $\limsup$ of empirical growth rates are, conditioned on
survival, bound\-ed by the limit growths rates of the w.f.-process and the
s.f.-process. This can be seen as a conditional
envelopment result with the w.f.- and the s.f.-policies as extreme
policies. If the $\liminf$ and $\limsup$ coincide, we can call the
limit $\gamma$ (without much abuse of terminology) the \textit{``Malthusian''} growth rate.

%s4.6.2 #&#
\subsubsection{Unconditional envelopment theorem}
We shall prove a stronger unconditional result: if there is a positive
survival probability for the process $(\Gamma_n)_n$, then, given that
the size of the process $(\Gamma_n)_n$ is sufficiently large, the
growth rate of that process dominates, with overwhelming probability,
that of the corresponding s.f.-society at all times $n\ge n_0$. This is
essentially the statement of Proposition~\ref{propprepborn} in
Section~\ref{chapboundproof}, and this allows us to deduce the
following envelopment theorem.

A few definitions are needed:
for any $L\in\N_0$, let $(S_n(L))_n$, $(\Gamma_n(L))_n$ and
$(W_n(L))_n$ denote, respectively, the s.f.-process, an arbitrary RDBP,
and the w.f.-process, each starting with initial size $L$. Hence,
$S_n=S_n(1)$, $\Gamma_n=\Gamma_n(1)$ and $W_n=W_n(1)$. Also let
$\theta
$ be defined as in Theorem~\ref{th2}.

%th4.13 #&#
%
\begin{theor}[(Envelopment theorem)]\label{thbounds}
Assume that $m(1-F(\theta))\ne1$ if $r\le m\mu$. Then,
\[
\p \Bigl[\lim_{n\to\infty}S_n(L)\le\lim
_{n\to\infty}\Gamma _n(L)\le\lim_{n\to\infty}W_n(L)
\Bigr]\xrightarrow{L\to\infty}1.
\]
Moreover, $q_{W}=1\Rightarrow q_\Gamma=1\Rightarrow q_S=1$.
\end{theor}

\begin{pf}See Section~\ref{chapboundproof}. (The same result holds
in the multiparameter case.)
\end{pf}

Proposition~\ref{propprepborn} in Section~\ref{chapboundproof} will
give more precise information about the lower bound. Such bounds are of
considerable theoretical interest, and, as we shall now see, they are
also serving as useful directives for individuals who have decided to
adapt a specific type of society. Indeed, if the probability laws of
the random variables $(D_n^k)_{n,k}$, $(X_n^k)_{n,k}$ and
$(R_n^k)_{n,k}$ are fixed up to their mean $m$, $\mu$ and $r$,
respectively, then it is in practice interesting to determine the \textit{critical mean resource production} $r_{\Gamma,c}(m,\mu)$, say, relative
to the RDBP $(\Gamma_n)_n$, that is, the value such that
\[
q_\Gamma=1\qquad\mbox{if } r<r_{\Gamma,c}(m,\mu)\quad\mbox {and}
\quad q_\Gamma<1\qquad\mbox{if } r>r_{\Gamma,c}(m,\mu).
\]
By Theorem~\ref{thbounds}, the following can be deduced:

%co4.14 #&#
%
\begin{cor}[(Critical curves for survival)]
For all $m,\mu$, we have
\[
r_{W,c}(m,\mu)\le r_{\Gamma,c}(m,\mu)\le r_{S,c}(m,\mu).
\]
\end{cor}

Therefore, the study of the two extreme RDBPs gives highly
relevant information about general RDBPs, without having to understand
every single possible policy (see examples in Section~\ref{exs}).
As we have seen, the computation of the critical mean resource
production even shows more. The point is that the mean claim value
plays only one part but that the resource claim distribution function
$F$ (which determines the mean, of course) plays itself an important
part. Hence society may try to take influence on individuals to settle,
under a fixed mean claim $\mu$, for a distribution $F$ which favors survival.\looseness=-1

%re4.15 #&#
%
\begin{rem}
If $D$, $X$ and $R$ were not assumed to be independent, Theorem~\ref
{thbounds} would in general not remain true: the w.f.-policy and the
s.f.-policy would {a priori} not remain extreme policies. We could
then naturally wonder how different dependence patterns yield different
extreme policies. Such questions may attract interest for further studies.
\end{rem}

%s5 #&#
\section{Examples}\label{exs}
We now give examples. It will be interesting to notice that the
critical mean resource production for a w.f.-process turns out to be
lower than one would intuitively expect.

\begin{longlist}[iii)]
\item[(i)] Let $F$ be the uniform distribution function on $(0,d)$,
say. Then $\mu=d/2$. As in Theorems \ref{th1}(a) and \ref{th2}(a), let
$r<m\mu=md/2$ and suppose that $F(r/k)>0$ for some $k\ge2$ with
$p_k>0$.

First, focus on the corresponding w.f.-process. The value of $\tau$ [see
equation~(\ref{tauequ})] is thus determined by:
\[
\frac{r}m=\int_0^\tau x\, \mathrm{d}F(x)=
\int_0^\tau\frac{1}{d}x\,\mathrm {d}x=
\frac{\tau^2}{2d},
\]
which yields $\tau=\sqrt{2dr/m}$. Therefore, $F(\tau)=\sqrt{2r/(md)}$.
The critical mean resource production $r_{w,c}$ is thus determined by
\[
mF(\tau)=1\quad\Longleftrightarrow\quad\sqrt{2r_{w,c}m/d}=1
\]
which implies $r_{w,c}=d/2m=\mu/m$. Note that, $r_{w,c}=d/2m$, which is
for larger $m$ not far from the expected value of the smallest order
statistic of claims of $2m$ descendants. With such a low creation of
resources, the f.c.f.s.-process or s.f.-process would die out very quickly,
as we shall see now.

For the s.f.-process we need $\theta$ defined by [see equation~(\ref{thetaequ})]:
\[
\frac{r}m=\int_\theta^d
\frac{1}{d}x\,\mathrm{d}x=\frac{d^2-\theta^2}{2d},
\]
and straightforward calculations yield $r_{s,c}=d(1-1/2m)=\mu(2-1/m)$.

We note that the critical mean resource production is now $2m-1$ times
higher than for the corresponding w.f.-process. Hence, if individuals
living in the \mbox{w.f.-}society on the critical value of creation and want to
change to the s.f.-society, then they must increase their average
resource creation by a factor $2m-1$ to be able to survive in the long
run, that is, an enormous difference. For instance, if $m=3$, the
critical resource creation mean must increase by factor five to
maintain a chance of survival!
Comparing with the corresponding critical mean resource production for
a f.c.f.s.-process, $r_{u,c}=\mu$, gives
\[
r_{u,c}-r_{w,c}=\mu(1-1/m)=r_{s,c}-r_{u,c}.
\]
Figure~\ref{graph2} compares the behavior of $r_{w,c}$ and $r_{s,c}$ as
functions of $m$. The area between the two curves corresponds to a
\textit{control area}, where the population can survive or get extinct
depending on the policy.
%
%f1 #&#
%
\begin{figure}%[!htp]

\includegraphics{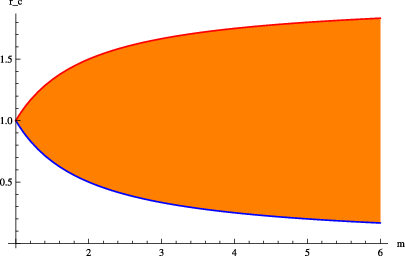}

\caption{For $d=2$ $(\mu=1)$, the critical mean resource
productions $r_{w,c}$ and $r_{s,c}$ are plotted (in blue and in red,
resp.), as functions of $m$.}\vspace*{-2pt}\label{graph2}
\end{figure}
\item[(ii)] Of course, we realize that the uniform distribution pushes
the largest and the smallest order statistics far apart. Therefore, it
is informative to look also at a case when the resource claim
distribution is more concentrated around its mean, as, for instance, in
the case of a beta distribution on $(0,1)$, with parameters $a$ and
$b$, say.
The distribution function is then defined on $(0,1)$ by the regularized
incomplete beta function: $F(x)=I_{a,b}(x)$. The mean resource claim is
given by $\mu=\frac{a}{a+b}$. As in Theorems \ref{th1}(a) and \ref{th2}(a), let $r<m\mu=\frac{am}{a+b}$ and suppose that $F(r/k)>0$ for
some $k\ge2$ with $p_k>0$.

First, focus on the corresponding w.f.-process. The value of $\tau$ [see
equation~(\ref{tauequ})] is determined by
\begin{eqnarray*}
\frac{r}m &=& \int_0^\tau x\,\mathrm{d}F(x)=
\int_0^\tau\frac
{x^a(1-x)^{b-1}}{B(a,b)}=
\frac{B(a+1,b)}{B(a,b)}I_{a+1,b}(\tau)
\\
&=& \frac{a}{a+b}I_{a+1,b}(
\tau),
\end{eqnarray*}
which yields $\tau=I^{-1}_{a+1,b} (\frac{r}m\frac{a+b}a )$. The
critical mean resource production $r_{w,c}$ is thus defined by
\[
mF(\tau)=1\quad\Longleftrightarrow\quad m I_{a,b} \biggl(I^{-1}_{a+1,b}
\biggl(r_{w,c}\frac{a+b}{am} \biggr) \biggr)=1,
\]
which implies\vspace*{-2pt}
\[
r_{w,c}=\frac{am}{a+b}I_{a+1,b}\bigl(I_{a,b}^{-1}(1/m)
\bigr).
\]

Now look at the corresponding s.f.-process. The value of $\theta$ [see
equation~(\ref{thetaequ})] is determined by
\[
\frac{r}m=\int_\theta^1 x\,\mathrm{d}F(x)=
\frac
{a}{a+b}\bigl(1-I_{a+1,b}(\theta)\bigr),
\]
which yields $\theta=I_{a+1,b}^{-1} (1-\frac{r}m\frac{a+b}a )$.
Straightforward calculations then give the critical mean resource
production $r_{s,c}=\frac
{am}{a+b}(1-I_{a+1,b}(I_{a,b}^{-1}(1-1/m)))$.

%
%f2 #&#
%
\begin{figure}[b]

\includegraphics{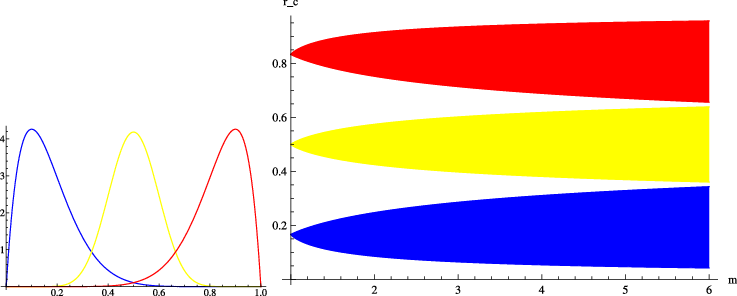}

\caption{The critical mean resource productions
$r_{w,c}$ and $r_{s,c}$ are plotted as functions of $m$, in the case of
a $B(a,b)$ resource claim distribution, for typical values of $(a,b)$:
$(2,10)$ in blue, $(14,14)$ in yellow and $(10,2)$ in red.}
\label{graph3}
\end{figure}

Observing that $I_{\alpha,\beta}(1-x)=1-I_{\beta,\alpha}(x)$, we deduce
that $I_{\alpha,\beta}^{-1}(1-x)=1-I_{\beta,\alpha}^{-1}(x)$. The
formula for $r_{s,c}$ can thus be rewritten as
\[
r_{s,c}=\frac{am}{a+b} I_{b,a+1} \bigl(I_{b,a}^{-1}
(1/m ) \bigr).
\]
For a f.c.f.s.-process, the corresponding critical mean resource production
simply reads $r_{u,c}=\mu$. Further, for large $m$, we can use the
approximation
\[
I_{\alpha,\beta}(z)=\frac{z^\alpha}{B(\alpha,\beta)} \biggl(\frac{1}\alpha +
\frac{1-\beta}{\alpha+1}z+O\bigl(z^2\bigr) \biggr);
\]
see, for example, \citet{Pearson68}. Straightforward calculations then
give, at leading order,
\[
r_{w,c}=\frac{a}{a+1} \biggl(\frac{a}{m}B(a,b)
\biggr)^{1/a}+O\bigl(m^{-2/a}\bigr),
\]
and
\[
r_{s,c}=1+\frac{b}{b+1}\bigl(b-a(b+1)\bigr) \biggl(
\frac{b}mB(b,a) \biggr)^{1/b}+O\bigl(m^{-2/b}\bigr).
\]
Figure~\ref{graph3} shows the critical areas in some typical cases, as
the peak is centered, moved to the left or to the right. Figure~\ref
{graph4} shows, in the centered case, how the critical area narrows as
the dispersion around the peak diminishes.
%f3 #&#
%
\begin{figure}%[b]

\includegraphics{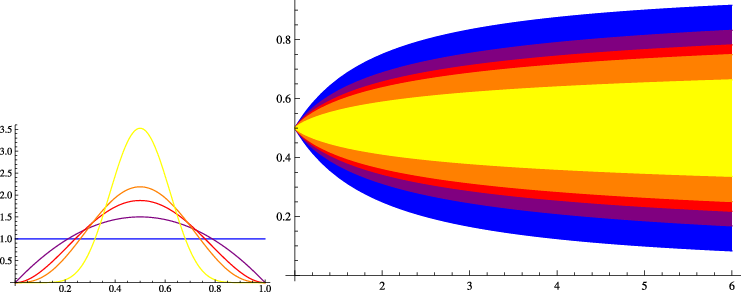}

\caption{The critical mean resource productions
$r_{w,c}$ and $r_{s,c}$ are plotted as functions of $m$, in the case of
a $B(a,b)$ resource claim distribution, for different symmetric values
of $(a,b)$: $(1,1)$ in blue, $(2,2)$ in pink, $(3,3)$ in red, $(4,4)$
in orange and $(10,10)$ in red.}\label{graph4}
\end{figure}

\item[(iii)] In the third example we choose a case where resource
claims are not bounded. Our results in the s.f.-case can therefore not be
used directly but it is interesting to see what happens to the
corresponding w.f.-process. Let $F$ be the distribution function of an
exponential random variable with parameter $\lambda$. The mean resource
claim is given by $\mu=1/\lambda$. As in Theorem \ref{th1}(a), let
$r<m\mu=m/\lambda$ and suppose that $F(r/k)>0$ for some $k\ge2$ with $p_k>0$.
The value of $\tau$ is determined by [see equation~(\ref{tauequ})]
\[
\frac{r}m=\int_0^\tau x\,\mathrm{d}F(x)=
\int_0^\tau\lambda xe^{-\lambda
x}\,\mathrm{d}x=
\frac{1}\lambda-e \biggl(\tau+\frac{1}\lambda
\biggr)e^{-\lambda
(\tau+1/\lambda)}
\]
which yields $\tau=-\frac{1}\lambda (1+\mathrm{W} [-\frac
\lambda
e (\frac{1}\lambda-\frac{r}m ) ] )$, where
$\mathrm
{W}[\cdot]$ denotes the Lambert W function; see, for example, \citet{Corless96}. The critical mean resource production $r_{w,c}$ is thus
determined by
\[
mF(\tau)=1\quad\Longleftrightarrow\quad m \biggl(1-\exp \biggl(\mathrm {W} \biggl[-
\frac\lambda e \biggl(\frac{1}\lambda-\frac{r_{w,c}}m \biggr)
\biggr]+1 \biggr) \biggr)=1.
\]
After simplifications, we get $r_{w,c}=\frac{1}\lambda
(1-(m-1)\log
{ (\frac{m}{m-1} )} )$, where we recall that $\frac{1}\lambda
=\mu$. For larger $m$, this becomes $r_{w,c}\approx1/(2\lambda m)=\mu
/(2m)$, that is, about one half of what is required for $U[0,1]$-claims.
\end{longlist}
%

%s6 #&#
\section{Proofs}\label{preuves}
%s6.1 #&#
\subsection{Preliminary results}\label{prmarkov}
We first prove Lemma~\ref{lemunimod}.

\begin{pf*}{Proof of Lemma~\ref{lemunimod}}
The property of $ M (t,(x_k)_{k=1}^t,s )$ being increasing in
$s$ is evident from the definition. To see unimodality in $t$, let
$I(t,s)=\mathbh1 \{\sum_{j=1}^tx_{j,t}\le s \}$. Then
(\ref{defM}) can be written as
\[
M \bigl(t,(x_k)_{k=1}^t,s \bigr)=I(t,s)t+
\bigl(1-I(t,s)\bigr) \sup \Biggl\{1\le k\le t\dvtx \sum_{j=t-k+1}^tx_{j,t}
\le s \Biggr\},
\]
where $\sup\varnothing:=0$.
This sum is clearly linearly increasing in $t$ as long as $I(t,s)=1$;
it is decreasing in $t$ as soon as $I(t,s)=0$, since the $k$ largest
order statistics are increasing with the sample size $t$. Hence, if
$t_1,t_2\in\N_0$, then, for fixed $s$, $M(t,(x_k)_{k=1}^t,s)$ is either
monotone increasing or monotone decreasing on $[t_1,t_2]$, or else
takes its maximum somewhere in $(t_1,t_2)$. This means that $M$ is
cap-unimodal in $t$, and hence the minimum
\[
\min_{t\in[t_1,t_2]} M \bigl(t,(x_k)_{k=1}^t,s
\bigr)
\]
is assumed in either $t_1$ or $t_2$.
Finally, the estimate for the corresponding maximum on $[t_1,t_2$] is
evident from the definition
of $M (t,(x_k)_{k=1}^t,s )$.
\end{pf*}

We now prove Proposition~\ref{markov}.

\begin{pf*}{Proof of Proposition~\ref{markov}}
Let $\pi$ be some policy (in the sense of Definition~\ref{defpol}).
Given $\Gamma_n$, the distributions of $D_n(\Gamma_n)$ and
$R_n(\Gamma
_n)$ are independent from $\Gamma_1,\ldots,\Gamma_{n-1}$, so that
\[
\Gamma_{n+1}=Q^\pi \bigl(D_n(
\Gamma_n),\bigl(X_n^k\bigr)_{k=1}^{D_n(\Gamma
_n)},R_n(
\Gamma_n) \bigr)
\]
is independent of $\Gamma_1,\ldots,\Gamma_{n-1}$, given $\Gamma_n$.
Thus, $(\Gamma_n)_n$ is a Markov process.
Now note that, since $Q^\pi(0,\varnothing,s)=0$ and $D_n(0)=0$ for all
$n\in\mathbb N$, we have $\{\Gamma_n=0\}\subset\{\Gamma_{n+1}=0\}$ so
that $0$ is an absorbing state for the process $(\Gamma_n)_n$.
Moreover, since
%
%e12 #&#
%
\begin{equation}
\Gamma_{n+1}=Q^\pi \bigl(D_n(
\Gamma_n),\bigl(X_n^k\bigr)_{k=1}^{D_n(\Gamma
)},R_n(
\Gamma_n) \bigr)\le D_n(\Gamma_n),
\end{equation}
it follows that
%
%e13 #&#
%
\begin{equation}
\p[\Gamma_{n+1}=0|\Gamma_n]\ge P\bigl[D_n(
\Gamma_n)=0|\Gamma _n\bigr]=p_0^{\Gamma_n},
\end{equation}
where the last equality holds because of the assumption of independent
reproduction. Therefore, the absorbing state $0$ is accessible from any
state $s\in\N$, with at least probability $p_0^s>0$. The state $0$ is
thus the only absorbing state, and, as $(\Gamma_n)_n$ is a Markov
process, we conclude
%
%e14 #&#
%
\begin{equation}
\p[0<\Gamma_n\le s\mbox{ i.o.}]=0\qquad\forall s\in
\N_0.
\end{equation}
The same arguments immediately adapt to multiparameter policies.
\end{pf*}

We now turn to the proof of Proposition~\ref{propboundM}. The idea is
that we compare the behavior of $(W_n)_n$ in each step $n$ with a
process consisting of $W_n$ i.i.d. versions of a weakest-first process
starting with one individual.

\begin{pf*}{Proof of Proposition~\ref{propboundM}}
Let $L\in\N_0$ and let $W_n^{(1)},\ldots,W_n^{(L)}$ be $L$ i.i.d.
copies of a weakest-first process. We then have the following
(super\-additivity-type) inequality, namely, for all $ n,  k \in\N_0$,
%
%e15 #&#
%
\begin{eqnarray}
\label{eqinegtriang}
&& \p[W_n\le k | W_0=L]
\nonumber\\[-8pt]\\[-8pt]
&&\qquad \le\p
\bigl[W_n^{(1)}+\cdots+W_n^{(L)}\le k
| W_0^{(1)}=\cdots=W_0^{(L)}=1\bigr],\nonumber
\end{eqnarray}
which we shall prove first.

To see this, we begin with the case $n=1$.

The case $L=1$ is trivial; hence suppose $L>1$. The LHS of~(\ref{eqinegtriang}) becomes, by an additional conditioning on $D_1(L)\sim D(L)$,
\begin{eqnarray*}
&& P[W_1\le k|W_0=L]
\\
&&\qquad =P\bigl[D(L)\le k\bigr]+P
\bigl[W_1\le k|D_1(L)> k,W_0=L\bigr] P
\bigl[D(L)> k\bigr],
\end{eqnarray*}
where we have used the facts that the distribution of $D(L)$ depends
only on $L$ (and not on the generation number), and also that $W_1$
cannot possibly exceed $D_1(W_0)=D_1(L)$.

Now suppose we do the same conditioning on the RHS of~(\ref{eqinegtriang}), that is, for the offspring of the $L$
partitioned processes. The first term $P[D(L)\le k]$ is then the same
on both sides, since, as before, reproduction of individuals is
independent. Hence, subtracting equal terms on both sides we can now
limit our interest to the corresponding second term with more than $k$
offspring.

The distributions of the total created resource space and of the claims
are by definition the same on both sides; therefore it suffices to look
for the moment at the influence of the order statistics of claims in a
\textit{fixed} sequence of claims on a \textit{fixed} resource space $R$, say.

In the LHS model of (\ref{eqinegtriang}), the resource space $R$ is
global (i.e., united) because all descendants from the different
families contribute to a common resource space. In the RHS model of
(\ref{eqinegtriang}), this resource space is, however, local (i.e.,
compartmented). On the LHS the count of individuals to stay is
therefore the count of the globally smallest order statistics of claims
which can be successively accommodated by $R$ whereas on the RHS the
count is on the locally smallest order statistics of claims. The
latter, put in increasing order, are a subsequence of the sequence of
claims in increasing order. Hence the RHS count cannot exceed the lhs count.

Passing from the counting argument to the corresponding probability
measures on both sides proves~(\ref{eqinegtriang}) for $n=1$, that is,
under the condition $W_0=L$ the number of descendants staying in the
population is stochastically larger than $W_1^{(1)}+\cdots+W_1^{(L)}$,
that is, for all $j\in\N_0$,
\[
\p[W_1\geq j| W_0=L]\geq\p\bigl[W_1^{(1)}+
\cdots+W_1^{(L)}\geq j | W_0^{(1)}=
\cdots=W_0^{(L)}=1\bigr].
\]

But now, we can iterate this argument. Clearly inequality~(\ref
{eqinegtriang}) must hold in particular if we replace, on the RHS only,
the number $L$ by some $L'$ with $L'\le L$. Hence the stochastic order
is maintained through the next generation, and thus, by recurrence,
through all generations. This implies that~(\ref{eqinegtriang}) is
true for all $n\in\N_0$.

Finally, choosing $k=0$ in~(\ref{eqinegtriang}) and taking the limit
for $n\to\infty$, we obtain by independence of the processes
$(W_n^{(j)})_n$ that
%
%e16 #&#
%
\begin{eqnarray}
&& \p[W_n\to0 | W_0=L]
\nonumber\\[-8pt]\\[-8pt]
&&\qquad \le\p\bigl[W_n^{(1)}
\to0,\ldots,W_n^{(L)}\to0 | W_0^{(1)}=
\cdots=W_0^{(L)}=1\bigr]=q_W^L,\nonumber
\end{eqnarray}
which completes the proof.
\end{pf*}

%re6.1 #&#
%
\begin{rem}The superadditivity-type inequality~(\ref{eqinegtriang}) is
in general no longer correct if the w.f.-process is replaced by other
RDBPs. Indeed, a very large claim may force on the RHS all the
offspring of one subpopulation to leave, but this effect stays still
local whereas it may be large on the global LHS. This exemplifies at
the same time the adherent difficulty in estimating extinction
probabilities for arbitrary policies.
\end{rem}

%s6.2 #&#
\subsection{Uniform bounds}\label{prbounds}
We first prove Proposition~\ref{propbounds}.

\begin{pf*}{Proof of Proposition~\ref{propbounds}}
Let $\pi$ be any policy (the same arguments immediately adapt to
multiparameter policies). First note that, by definition of $N$~and~$M$,
%
%e17 #&#
%
\begin{eqnarray}
M \bigl(t,(x_k)_{k=1}^t,s \bigr)\le
Q^\pi \bigl(t,(x_k)_{k=1}^t,s \bigr)
\le N \bigl(t,(x_k)_{k=1}^t,s \bigr)
\nonumber\\[-8pt]\\[-8pt]
\eqntext{\forall t,\ \forall (x_k)_{k=1}^t,\ \forall s.}
\end{eqnarray}
We shall now show by induction that if $(\Gamma_n)_n$ is the RDBP
controlled by $\pi$, it follows that $\Gamma_n\le W_n$ a.s. for all
$n$, given $W_0=\Gamma_0=1$. Indeed, it is true at any time at which
the two processes have the same number of individuals, and hence for $n=0$.

Now, if it is true for some $n$, we deduce that, a.s.,
%
%e18 #&#
%e19 #&#
%e20 #&#
%
\begin{eqnarray}
\Gamma_{n+1}&=&Q^\pi \bigl(D_n(
\Gamma_n),\bigl(X_n^k\bigr)_{k=1}^{D_n(\Gamma
_n)},R_n(
\Gamma_n) \bigr)
\\
&\leq&  N \bigl(D_n(\Gamma_n),\bigl(X_n^k
\bigr)_{k=1}^{D_n(\Gamma_n)},R_n(\Gamma _n)
\bigr)
\\
&\leq&  N \bigl(D_n(W_n),\bigl(X_n^k
\bigr)_{k=1}^{D_n(W_n)},R_n(W_n)
\bigr)=W_{n+1},
\end{eqnarray}
as the mapping $(t,s)\mapsto N(t,(x_k)_{k=1}^t,s)$ is increasing in
both arguments. Hence the inequality is also true for $n+1$.

It follows that $\p[\Gamma_n \le W_n]=1$ for all $n$. Since the
limiting extinction probabilities of $(\Gamma_n)_n$ and $(W_n)_n$ must
exist, we must also have $q_W\le q_\Gamma$.
\end{pf*}

We now give an explicit counterexample showing that it is in general
not true that $S_n\le\Gamma_n$ a.s. for all $n$, given $S_0=\Gamma
_0=1$. The underlying idea was already explained in Section~\ref{chapnolow}.

\subsubsection*{Counterexample}
We have assumed $p_k>0$ for some $k\ge2$ [see regularity assumption
(ii) in Section~\ref{chapnatregcond}]; to fix ideas, assume that
$p_3>0$ (the argument can be adapted in any case). Then consider the
deterministic policy $\pi$ given by
%
%e21 #&#
%
\begin{equation}
\pi_t\bigl((x_k)_{k=1}^t\bigr) (j)=
\cases{ \sigma(3),&\quad if $j=1$ and $t\ge3$,
\vspace*{2pt}\cr
\sigma(1),&\quad
if $j=2$ and $t\ge3$,
\vspace*{2pt}\cr
\sigma(2),&\quad if $j=3$ and $t\ge3$,
\vspace*{2pt}\cr
\sigma(j),&
\quad otherwise,} %
\end{equation}
where $\sigma$ is the permutation such that $x_{\sigma(1)}\ge\cdots
\ge
x_{\sigma(t)}$ [i.e., by definition, $\sigma=\pi
^S_t((x_k)_{k=1}^t)$]. Let, for example,
\begin{eqnarray*}
D_0^1&=&3,\qquad
X_0^1>X_0^2>X_0^3,\qquad
X_0^1+X_0^3<R_0^1<X_0^1+X_0^2,
\\
D_1^1&=&D_1^2=3,\qquad
X_1^1+X_1^2+X_1^3\le R_1^1
\end{eqnarray*}
and then
\[
X_1^4,X_1^5,X_1^6>R_1^1+R_1^2.
\]
These events will occur simultaneously with positive probability, as
$p_3>0$. But then we immediately see that $\Gamma_2=0<3=S_2$ in this case.

%s6.3 #&#
\subsection{Extinction criterion for the w.f.-society}\label{prth1}
In this section, we will prove Theorem~\ref{th1}. In these proofs, we
will repeatedly make use of the following lemma, which we shall prove first.

%le6.2 #&#
%
\begin{Lem}\label{lemme}
Let $X_1,X_2,\ldots$ be i.i.d. real-valued nonnegative random
variables with mean $\mu<\infty$ and continuous distribution function
$F$. Further, let $(\Phi_n)_n$ be a sequence of integer-valued random
variables with $\Phi_n\to\infty$ a.s. as $n\to\infty$, and let
$(\Psi
_n)_n$ be a sequence of real random variables with $\Psi_n\to\infty$
a.s. as $n\to\infty$. Suppose that $\Psi_n/\Phi_n\to\rho$ a.s. with
$0<\rho\le\mu$, and that $\tau$ is the solution of
\[
\int_0^\tau x\,\mathrm{d}F(x)=\rho.
\]
Let $N$ be defined by~(\ref{defN}). Then $N(\Phi_n,(X_k)_{k=1}^{\Phi
_n},\Psi_n)/\Phi_n\to F(\tau)$ a.s.

Moreover, if the random variables $X_1,X_2,\ldots$
are bounded, then we have an analogous result for $M$ defined in (\ref
{defM}): defining $\theta$ as the solution of
\[
\int_\theta^b x\,\mathrm{d}F(x)=\rho,
\]
then $M(\Phi_n,(X_k)_{k=1}^{\Phi_n},\Psi_n)/\Phi_n\to1-F(\theta)$ a.s.
\end{Lem}

\begin{pf}
Since $\Psi_n/\Phi_n\to\rho>0$ a.s. as $n\to\infty$, we have
%
%e22 #&#
%
\begin{equation}
\forall\varepsilon>0\qquad  \p \biggl[\mathop{\operatorname{sup}}_{n\ge
b}\biggl
\llvert \frac{\Psi_n}{\Phi_n}-\rho\biggr\rrvert <\varepsilon \biggr]\to1\qquad
\mbox{a.s. as $b\to\infty$.}
\end{equation}

To simplify notation, we write $N(t,s):=N(t,(X_k)_{k=1}^t,s)$. (Note
that this simplification is here admissible since the distribution of
the string of claims depends only on $t$.)

As the function $N(t,s)$ is stochastically increasing in $s$, we deduce
from almost-sure convergence of $\Phi_n/\Psi_n$ to $\rho$ that, for all
$0\le\varepsilon<\rho$ and all $\delta>0$, the inequalities
%
%e23 #&#
%
\begin{equation}
\label{ineg} \frac{N(\Phi_n,(\rho-\varepsilon)\Phi_n)}{\Phi_n}\le\frac{N(\Phi
_n,\Psi
_n)}{\Phi_n}\le\frac{N(\Phi_n,(\rho+\varepsilon)\Phi_n)}{\Phi_n}
\end{equation}
must hold (simultaneously), for all $n$ sufficiently large, with
probability at least $1-\delta$.

We now use Theorem 2.2 (on page 615) of \citet{Bruss91}. This theorem
[refining a result of \citet{Coffman87}] implies that
%
%e24 #&#
%
\begin{equation}
\frac{N(n, s_n)}{n}\to F\bigl(\tau(s)\bigr)\qquad\mbox{a.s. as $n\to\infty$,}
\end{equation}
where $\tau(s)$ is the solution of
%
%e25 #&#
%
\begin{equation}
\int_0^{\tau(s)}x \,\mathrm{d}F(x)=\lim
_{n\to\infty}\frac{s_n}n=:s,
\end{equation}
provided that the latter limit exists and satisfies $0<s\le\mu=\E[X]$.

Now note that $\tau(\cdot)$ is continuous on $(0,\mu)$ because $F$ is
assumed to be continuous on its support.

As $\Phi_n\to\infty$ a.s., the left-hand side variable of (\ref{ineg})
must converge a.s. to $F(\tau(\rho-\varepsilon))$ and the right-hand side
variable of (\ref{ineg}) a.s. to $F(\tau(\rho+\varepsilon))$. Since
$\varepsilon>0$ is arbitrary and
\[
\lim_{\varepsilon\to0^+}F\bigl(\tau(\rho\pm\varepsilon)\bigr)=F\bigl(\tau(
\rho)\bigr)
\]
by continuity of $F(\cdot)$ and $\tau(\cdot)$, the first part of the
lemma is proved.

The second part of the lemma, that is, the statement that
\[
M\bigl(\Phi_n,(X_k)_{k=1}^{\Phi_n},
\Psi_n\bigr)/\Phi_n\to1-F(\theta)\qquad\mbox{a.s.},
\]
can now be proved similarly, using Theorem 2.3 of \citet{Bruss91}. Note
that we need here, as stated, the assumption that the resource claims
are bounded, since the cited Theorem 2.3 may not hold otherwise.
\end{pf}

We can now prove Theorem~\ref{th1}.

\begin{pf*}{Proof of Theorem~\ref{th1}}
We first prove statement (a).
Suppose $r\le m\mu$ and $mF(\tau)<1$, and let
%
%e26 #&#
%
\begin{equation}
\label{eqw1} W^\infty(\Omega)=\bigl\{\omega\in\Omega\dvtx W_n(
\omega)\to\infty\mbox{ as $n\to \infty$}\bigr\}.
\end{equation}
In the following, we can use the shorthand notation
$N_n(t,s):=N(t,(X_n^k)_{k=1}^t,\break s)$, where again the index $n$ will be
dropped only if the distribution is used.

Now look at
%
%e27 #&#
%
\begin{equation}
\label{ut1} \E[W_{n+1}|W_n=w]=\E\bigl[N_n
\bigl(D_n(w),R_n(w)\bigr)|W_n=w\bigr].
\end{equation}
Since $R_n(w)/w\to r$ a.s. as $w\to\infty$ and $D_n(w)/w\to m$ a.s. as
$w\to\infty$ and $m>0$, we have $R_n(w)/D_n(w)\to\rho:=r/m$ a.s.
According to Lemma 4.1 (on page 622) of \citet{Bruss91}, there exists a
sequence $T_w\to\tau$ a.s. as $w\to\infty$ with
%
%e28 #&#
%
\begin{equation}
\label{ut2} \E\bigl[N_n\bigl(D_n(w),R_n(w)
\bigr)|W_n=w\bigr]\le\E\bigl[D_n(w)\bigr]F(T_w),
\end{equation}
where $\tau$ is the solution of
\[
\int_0^\tau x \,\mathrm{d}F(x)=\rho=
\frac{r}{m}.
\]
Since $F$ is continuous we can find, for each $\varepsilon>0$, a value
$w=w(\varepsilon)$ such that $F(T_v)\le F(\tau+\varepsilon)$ for all $v\ge
w$. Thus, from equations (\ref{ut1}) and (\ref{ut2}),
%
%e29 #&#
%
\begin{equation}
\E[W_{n+1}|W_n=v]\le mv F(\tau+\varepsilon),\qquad v\ge w.
\end{equation}
Hence we get
%
%e30 #&#
%e31 #&#
%e32 #&#
%
\begin{eqnarray}
\E[W_{n+1}|W_n\ge w]&=&\sum_{v=w}^\infty
\p[W_n=v|W_n\ge w] \E [W_{n+1}|W_n=v]
\\
&\leq&  m F(\tau+\varepsilon) \sum_{v=w}^\infty v
\p[W_n=v|W_n\ge w]
\\
&=&m F(\tau+\varepsilon) \E[W_n|W_n\ge w].\label{eqw2}
\end{eqnarray}
Since $mF(\tau)<1$, we can choose, again by continuity of $F$, a
positive value $\varepsilon$ sufficiently small such that $mF(\tau
+\varepsilon
)<1$. The latter implies then that $(\E[W_n])_n$ must be bounded.
Consequently, Proposition~\ref{markov} implies that\break $\p[W^\infty
(\Omega
)]=0$, or equivalently, since $(W_n)_n$ is Markovian, $q_W=1$.

This proves the first part of Theorem~\ref{th1}(a).

To see the second part of Theorem~\ref{th1}(a), we now suppose that
$r\le m\mu$ and $mF(\tau)>1$.
Recall that we had supposed that any finite level can be reached with
a strictly positive probability [see regularity assumption~(iii) in
Section~\ref{chapnatregcond}]. Therefore it suffices to show that, for
$w$ sufficiently large,
%
%e33 #&#
%
\begin{equation}
\label{equsuff} \exists \alpha\ge1 \qquad\mathop{\operatorname{liminf}}_{k\to
\infty}
\p \bigl[W_{n+k}\ge\alpha^{k}w|W_n\ge w
\bigr]>0.
\end{equation}
Let now
%
%e34 #&#
%
\begin{equation}
\label{eqw3} h(j,\alpha,w):=\p\bigl[W_{n+j}<\alpha
^jw|W_{n+j-1}\ge\alpha^{j-1}w\bigr].
\end{equation}
It follows that
%
%e35 #&#
%
\begin{eqnarray}
&& \p\bigl[W_{n+k}\ge\alpha^kw|W_n\ge w\bigr]
\nonumber\\[-8pt]\\[-8pt]
&&\qquad \ge
\bigl(1-h(k,\alpha,w)\bigr)\p \bigl[W_{n+k-1}\ge \alpha^{k-1}w|W_n
\ge w\bigr]\nonumber
\end{eqnarray}
and thus by recurrence on $k$ that
%
%e36 #&#
%
\begin{equation}
\p\bigl[W_{n+k}\ge\alpha^kw|W_n\ge w\bigr]\ge
\prod_{j=1}^k\bigl(1-h(j,\alpha,w)\bigr).
\end{equation}
Therefore a sufficient condition for (\ref{equsuff}) to hold is
%
%e37 #&#
%
\begin{equation}
\label{condsuff} \sum_{j=1}^\infty h(j,
\alpha,w)<\infty.
\end{equation}
Since $N_n(\cdot,\cdot)$, $D_{n+j}(\cdot)$ and $R_{n+j}(\cdot)$ are all
stochastically increasing in their arguments, we have
%
%e38 #&#
%e39 #&#
%e40 #&#
%
\begin{eqnarray}
h(j,\alpha,w)&=&\p\bigl[W_{n+j}<\alpha^jw|W_{n+j-1}
\ge\alpha^{j-1}w\bigr]
\\
&=&\p\bigl[N_{n+j}\bigl(D_{n+j}(W_{n+j-1}),
\nonumber\\[-8pt]\\[-8pt]
&&\hspace*{10pt} R_{n+j}(W_{n+j-1})
\bigr)<\alpha ^jw|W_{n+j-1}\ge\alpha^{j-1}w\bigr]\nonumber
\\
&\leq& \p\bigl[N_{n+j}\bigl(D_{n+j}\bigl(\bigl\lfloor
\alpha^{j-1}w\bigr\rfloor \bigr),R_{n+j}\bigl(\bigl\lfloor
\alpha^{j-1}w\bigr\rfloor\bigr)\bigr)<\alpha^jw
\bigr].\label{inequ}
\end{eqnarray}
Choose $\varepsilon>0$ such that $(m-\varepsilon)F(\tau)>1$, and put
$\alpha
=(m-\varepsilon)F(\tau)$. Then we have $p_0^\alpha<1$ so that
%
%e41 #&#
%
\begin{equation}
\sum_{j=1}^\infty p_0^{\lfloor\alpha^{j-1}w\rfloor}=
\sum_{j=1}^\infty\p \bigl[D_{n+j}
\bigl(\bigl\lfloor\alpha^{j-1}w\bigr\rfloor\bigr)=0\bigr]<\infty.
\end{equation}
Since the $D_{n+j}$ are independent random variables, it follows from
the Borel--Cantelli lemma that $\p[D_{n+j}(\lfloor\alpha
^{j-1}w\rfloor
)=0$ $j$-i.o.$]=0$. Therefore, for $j$ sufficiently large,
inequality (\ref{inequ}) is equivalent to
%
%e42 #&#
%
\begin{equation}
\label{equhLR} h(j,\alpha,w)\le\p\bigl[L_j^*<R_j^*\bigr],
\end{equation}
where the LHS variable is defined by
%
%e43 #&#
%
\begin{equation}
L_j^*=\frac{N_{n+j}(D_{n+j}(\lfloor\alpha^{j-1}w\rfloor
),R_{n+j}(\lfloor
\alpha^{j-1}w\rfloor))}{D_{n+j}(\lfloor\alpha^{j-1}w\rfloor
)}\label{defL}
\end{equation}
and the corresponding RHS variable by
%
%e44 #&#
%
\begin{equation}
R_j^*=\frac{\alpha^{j}w}{D_{n+j}(\lfloor\alpha^{j-1}w\rfloor
)}.\label{defR}
\end{equation}
Now let $\beta_j=\lfloor\alpha^{j-1}w\rfloor$.

First look at the random variables $D_{n+j}(\beta_j)/\beta_j\sim
D(\beta
_j)/\beta_j$ (and recall that, whenever we drop indices of variables in
our notation, then this means that we use information on their \textit{distributional} prescription only). $D(k)$ is hence distributed like a
sum of $k$ i.i.d. random variables with finite mean $m$ and finite
variance $\sigma_D^2$, say.

Recall that $\alpha>1$, and that $\beta_j/j\to\infty$ as $j\to
\infty$.
Therefore, it follows from the Hsu--Robbins theorem of complete
convergence [see Theorem~1 of~\citet{HsuRobbins}, or~\citet{Asmussen80}]
that $D(\beta_j)/\beta_j\to m$ completely as $j\to\infty$. Note that
complete convergence holds row-wise in the reproduction matrix since
all $D_k^j$ are i.i.d. and have finite variance. This implies
%
%e45 #&#
%
\begin{equation}
\forall\delta>0 \qquad\sum_{j=1}^\infty\p
\biggl[\biggl\llvert \frac
{D(\beta
_j)}{\beta_j}-m\biggr\rrvert >\delta \biggr]<
\infty.\label{equm}
\end{equation}
Further, since $\alpha^{j}w/\beta_j\to\alpha$ as $j\to\infty$ and
$D_{n+j}(\cdot)\sim D(\cdot)$, we obtain from (\ref{defR}) and (\ref{equm})
%
%e46 #&#
%
\begin{equation}
\label{equR} \forall\delta>0 \qquad\sum_{j=1}^\infty
\p \biggl[\biggl\llvert R_j^*-\frac\alpha m\biggr\rrvert >\delta
\biggr]<\infty.
\end{equation}
Second, to study the convergence of $L_j^*$ defined in (\ref{defL}) we
turn to Lemma~\ref{lemme} with $\Phi_j=D(\beta_j)$ and $\Psi
_j=R(\beta
_j)$. Since $\Phi_j/\beta_j\to m$ completely and $\Psi_j/\beta_j\to r$
completely (again by the Hsu--Robbins theorem), we have
%
%e47 #&#
%
\begin{equation}
\frac{\Psi_j}{\Phi_j}\to\rho=\frac{r}m\qquad\mbox{completely, as $j\to
\infty$.}
\end{equation}
Therefore, in particular, $\Psi_j/\Phi_j\to\rho$ a.s., so that the
conditions of Lemma~\ref{lemme} are satisfied. It follows that \emph{if} $L_j^*$ in (\ref{defL}) allows for a limit (in some sense) $l$,
say, then we must have $l=F(\tau)$, where $\tau$ is defined as in
Lemma~\ref{lemme}.

Using this and the Chernov-type estimates obtained by~\citet{Coffman87}
(see Theorems~2 and~3) with $a=j\delta$, we obtain after some
straightforward simplifications,
%
%e48 #&#
%
\begin{equation}
\p \biggl[\biggl\llvert \frac{N(\cdot,\cdot)}j-F(\tau)\biggr\rrvert >\delta \biggr]
\le 2e^{-({j\delta^2})/({4F(\tau)})}.
\end{equation}
Again $N_{n+j}(\cdot,\cdot)\sim N(\cdot,\cdot)$ and $\beta_j/j\to
\infty
$, and thus $L_j^*\to F(\tau)$ completely as $j\to\infty$. This implies
the convergence
%
%e49 #&#
%
\begin{equation}
\label{equL} \forall\delta>0 \qquad\sum_{j=1}^\infty
\p\bigl[\bigl|L_j^*-F(\tau)\bigr|>\delta \bigr]<\infty.
\end{equation}
Now choose $\delta=\frac{1}2|F(\tau)-\alpha/m|>0$. Note that the
event $\{
L_j^*<R_j^*\}$ can only occur if $|L_j^*-F(\tau)|>\delta$ or
$|R_j-\alpha/m|>\delta$. Therefore, from (\ref{equhLR}),
%
%e50 #&#
%e51 #&#
%
\begin{eqnarray}
h(j,\alpha,w)&\leq& \p\bigl[L_j^*<R_j^*\bigr]
\\
&\leq& \p\bigl[\bigl|L_j^*-F(\tau)\bigr|>\delta\bigr]+\p \biggl[\biggl\llvert
R_j^*- \frac\alpha m\biggr\rrvert >\delta \biggr]
\end{eqnarray}
so that, according to (\ref{equR}) and (\ref{equL}),
%
%e52 #&#
%
\begin{equation}
\sum_{j=1}^\infty h(j,\alpha,w)<\infty.
\end{equation}
This completes the proof of statement (a).

Statement (b) is obtained similarly (and more easily) using
Theorem 2.1 of~\citet{Bruss91}.

Finally, using Lemma~\ref{lemme}, it is clear that, if $r\le m\mu$ and
$mF(\tau)>1$, conditioning on survival, we have $W_{n+1}/W_{n}\to
mF(\tau)$ a.s. as $n\to\infty$, and thus, for any $\varepsilon,\delta>0$,
there exists some large $L_{\varepsilon,\delta}>0$ such that
%
%e53 #&#
%
\begin{equation}
\label{eqpremauv} \p \biggl[ \sup_{n\ge L_{\varepsilon,\delta}}\biggl\llvert
\frac
{W_{n+1}}{W_{n}}-mF(\tau)\biggr\rrvert <\delta \Big| \lim
_{n\to\infty
}W_n=\infty \biggr]\ge1-\varepsilon.
\end{equation}
This precisely means that, conditioning on survival, the w.f.-process
behaves more and more like a GWP with the modified reproduction mean
$\tilde m=mF(\tau)$.
\end{pf*}

%s6.4 #&#
\subsection{Extinction criterion for the s.f.-society (first part)}\label
{proofth2}
In this section, we are concerned with the proof of Theorem~\ref{th2}.
The first part of Theorem~\ref{th2}(a) can be obtained by similar
considerations as for Theorem~\ref{th1} (the extinction criterion for
the w.f.-society), the role of $\tau$ being now played by $\theta$,
defined by
\[
\int_\theta^bx \,\mathrm{d}F(x)=\frac{r}m,
\]
as in Theorem 2.3 of \citet{Bruss91}. However, here we need the
assumption that the resource claims are bounded as assumed in the
model, simply because the cited Theorem 2.3 may not hold otherwise.
(This was not needed in the case of the w.f.-society where we only used
that the claims have a finite variance.)

\begin{pf*}{Proof of Theorem~\ref{th2}\textup{(a)(i)}}
The proof of the first part of Theorem~\ref{th2}(a) follows the same
reasoning as for Theorem~\ref{th1} [see equations~(\ref{eqw1})--(\ref{eqw2})] and yields correspondingly that, for all
$\varepsilon>0$, there
exists a sufficiently large $s$ such that
%
%e54 #&#
%
\begin{equation}
\E[S_{n+1}|S_n\ge s]\le m\bigl(1-F(\theta-\varepsilon)\bigr)
\E[S_n|S_n\ge s].
\end{equation}
Choosing $\varepsilon$ sufficiently small such that $m(1-F(\theta
-\varepsilon
))<1$ shows that $(\E[S_{n+1}|S_n\ge s])_n$ is bounded, so that $(\E
[S_n])_n$ is bounded too. Hence, similarly as before, $q_S=1$.
\end{pf*}

For the other parts of Theorem~\ref{th2}, adapting the proof of
Theorem~\ref{th1} seems difficult. The major technical difficulty is
that we have to deal with the cap-unimodality of $M(\cdot,\cdot)$ in
its first argument [while $N(\cdot,\cdot)$ was increasing in both arguments].

Cap-unimodality implies that the minimum of $M(D(t),R(t))$ for given
$D(t)$ and $R(t) $ over an interval
$[t_1, t_2]$ must be taken on the border, but gives less information
about the corresponding maximum. The estimate $t_2-t_1$ for the
difference between the maximum and the minimum over $[t_1,t_2]$ (see Lemma~\ref
{lemunimod}) is here too crude. We therefore have to proceed
differently, and, in order to use arguments developed later, the rest
of this proof is postponed to Section~\ref{proofth2+}.

%s6.5 #&#
\subsection{Corollaries~\texorpdfstring{\protect\ref{cor1}}{4.6} and \texorpdfstring{\protect\ref{cor2}}{4.9}}\label{prcors}
We shall need the following two lemmas, which we prove first:

%
%le6.3 #&#
%
\begin{Lem}\label{lemcor1}
Assume $r>\mu$, and let $\tau$ be defined by
\[
\int_0^\tau x\,\mathrm{d}F(x)=\frac{r}m.
\]
Then $mF(\tau)>1$.
\end{Lem}

\begin{pf}
Let $a\ge0$ be the infimum of the support of $F$. As $F(x)>0$ for all
$x>a$ and as $\tau>a$ [because $\int_0^\tau x\,\mathrm{d}F(x)=\frac{r}m>\frac\mu m\ge0$], we deduce $F(\tau)>0$, and we can write
\[
mF(\tau)\int_0^\tau x\frac{\mathrm dF(x)}{F(\tau)}=r>\mu,
\]
or equivalently,
\[
mF(\tau)\E[X|X\le\tau]>\mu.
\]
However, $\E[X|X\le\tau]\le\E[X]=\mu$, so that the above inequality
cannot hold unless $mF(\tau)>1$.
\end{pf}

%
%le6.4 #&#
%
\begin{Lem}\label{lemcor2}
Assume $r<\mu$, and let $\theta$ be defined by
\[
\int_\theta^\infty x\,\mathrm{d}F(x)=\frac{r}m.
\]
Then $m(1-F(\theta))<1$.
\end{Lem}

\begin{pf}
Let $b\ge0$ be the least upper bound of claims. As $1-F(x)>0$ for all
$x<b$ and as $\theta<b$ [because $\int_\theta^\infty x\,\mathrm
{d}F(x)=\frac{r}m>0$], we deduce $1-F(\theta)>0$, and we can write
\[
m\bigl(1-F(\theta)\bigr)\int_\theta^\infty x
\frac{\mathrm dF(x)}{1-F(\theta
)}=r<\mu,
\]
or equivalently
\[
m\bigl(1-F(\theta)\bigr)\E[X|X>\theta]<\mu.
\]
However, $\E[X|X>\theta]\ge\E[X]=\mu$, so that the above inequality
implies $m(1-F(\theta))<1$.
\end{pf}

We can now prove Corollary~\ref{cor1}.

\begin{pf*}{Proof of Corollary~\ref{cor1}}
Using Lemma~\ref{lemcor1} and Theorem~\ref{th1}{(a)(ii)}, part~(a) is
immediate. It remains to prove part~(b). Using the Cauchy--Schwarz
inequality, we deduce directly from the definition of $\tau$,
%
%e55 #&#
%e56 #&#
%
\begin{eqnarray}
\frac{r}m&=&\int_0^\tau x\,\dd F(x)=\mu-\int
_\tau^bx\,\dd F(x)
\\
&\geq& \mu-\sqrt{1-F(\tau)}\sqrt{\int_\tau^bx^2\,\dd F(x)}
\ge\mu-\sqrt {1-F(\tau )}\sqrt{\E\bigl[X^2\bigr]},
\end{eqnarray}
so that, after straightforward simplifications, $mF(\tau)\le m-(m\mu
-r)^2/\break (m\E[X^2])$. Hence, if
\[
{\E\bigl[X^2\bigr]}<(m\mu-r)^2/\bigl(m(m-1)\bigr),
\]
we deduce $mF(\tau)<1$ and thus, by Theorem~\ref{th1}{(a)(i)}, we must
have $q_W=1$.
\end{pf*}

We now turn to the proof of Corollary~\ref{cor2}.

\begin{pf*}{Proof of Corollary~\ref{cor2}}
Using Lemma~\ref{lemcor2} and Theorem~\ref{th2}{(a)(i)}, part~(a) is
immediate.

It remains to prove part~(b).

Using the Cauchy--Schwarz inequality, we deduce directly from
the definition of $\theta$
%
%e57 #&#
%
\begin{eqnarray}
\qquad \frac{r}{m}&=&\int_\theta^bx\,\dd F(x)\le
\sqrt{1-F(\theta)}\sqrt{\int_\theta
^bx^2\,\dd F(x)}\le\sqrt{1-F(\theta)}\sqrt{\E
\bigl[X^2\bigr]}.
\end{eqnarray}
Hence, if ${\E[X^2]}<r^2/m$, we deduce
\[
m\bigl(1-F(\theta)\bigr)\ge r^2/\bigl(m\E\bigl[X^2\bigr]
\bigr)>1
\]
and thus, by Theorem~\ref{th2}{(a)(ii)}, we must have $q_S<1$.
\end{pf*}

%s6.6 #&#
\subsection{Conditional envelopment theorem}\label{prenvelop1}
We prove here Proposition~\ref{propenvelop1}.

\begin{pf*}{Proof of Proposition~\ref{propenvelop1}}
Assume that $\Gamma_n\to\infty$ is given, that is, that $\underline
\gamma>0$ is given. First compare the process $(\Gamma_n)_n$ with
$(S_n)_n$. We see from the corresponding counting functions $Q^\pi$ and
$M$ that
%
%e58 #&#
%e59 #&#
%
\begin{eqnarray}
\underline\gamma&=&\liminf_{n\to\infty}\frac{\Gamma_{n+1}}{\Gamma
_n} = \liminf
_{n\to\infty}\frac{Q^\pi_n (D_n(\Gamma
_n),R_n(\Gamma
_n) )}{\Gamma_n}
\\
&\geq& \liminf_{n\to\infty}\frac{M_n (D_n(\Gamma_n),R_n(\Gamma
_n)
)}{\Gamma_n},
\end{eqnarray}
where the last inequality follows from the definition of the function
$M$. Since the latter is increasing in the second argument and
cap-unimodal in the first argument, we define (neglecting the
floor--roof symbols which are here of no importance)
%
%e60 #&#
%
\begin{equation}
\qquad \widetilde M^\varepsilon_n(\Gamma_n)=\min \bigl
\{M_n\bigl((m-\varepsilon )\Gamma _n,(r-\varepsilon)
\Gamma_n\bigr),  M_n\bigl((m+\varepsilon)
\Gamma_n,(r -\varepsilon )\Gamma _n\bigr) \bigr\},
\end{equation}
so that, for all $0<\varepsilon<\min(r,m)$,
%
%e61 #&#
%
\begin{equation}
\label{plop1} \underline\gamma\ge\liminf_{n\to\infty}
\frac{1}{\Gamma
_n}\widetilde M^\varepsilon_n(\Gamma_n)=
\liminf_{n\to\infty}\frac{m}{D_n(\Gamma
_n)}\widetilde M^\varepsilon_n(
\Gamma_n).
\end{equation}
Now, let $\theta^-_\varepsilon$ and $\theta^+_\varepsilon$ be defined,
respectively, by
%
%e62 #&#
%
\begin{equation}
\int_{\theta_\varepsilon^-}^\infty x \,\dd F(x)=\frac{r-\varepsilon
}{m-\varepsilon
}\quad
\mbox{and}\quad\int_{\theta_\varepsilon^+}^\infty x \,\dd F(x)=
\frac
{r-\varepsilon}{m+\varepsilon}.
\end{equation}
Then, from Lemma~\ref{lemme}, we get
%
%e63 #&#
%
\begin{equation}
\label{plop2} \frac{\widetilde M^\varepsilon_n(\Gamma_n)}{D_n(\Gamma_n)}\ge\min \bigl\{ 1-F\bigl(\theta_\varepsilon^-
\bigr), 1-F\bigl(\theta_\varepsilon^+\bigr) \bigr\},
\end{equation}
for all sufficiently large $n$. Hence, using equations~(\ref{plop1})
and~(\ref{plop2}), as well as the continuity of $\theta_\varepsilon^-$ and
$\theta_\varepsilon^+$ as functions of $\varepsilon$, we deduce
%
%e64 #&#
%
\begin{equation}
\underline\gamma\ge m\bigl(1-F(\theta)\bigr).
\end{equation}

Second, we must compare $(\Gamma_n)_n$ with $(W_n)_n$. This is done
similarly and more easily because $N(\cdot, \cdot) $ is monotone
increasing in both arguments. This yields then directly the other
stated inequality $\bar\gamma\le mF(\tau)$ and the proof is complete.
\end{pf*}

%s6.7 #&#
\subsection{Unconditional envelopment theorem}\label{chapboundproof}
In this section, we are concerned with the proof of Theorem~\ref
{thbounds}. In the case when $r\le m\mu$, we define $\tau$ and
$\theta
$ as in the statement of Theorems~\ref{th1} and~\ref{th2}. When
$r>m\mu
$, these are not defined, but, to simplify notations, we then define
$F(\tau):=1$ and $1-F(\theta):=1$.

First of all, we shall need the following interesting
result, which is a far-reaching strengthening of equation~(\ref{eqpremauv}):

%
%th6.5 #&#
%
\begin{theor}\label{thborneas}
For any $\varepsilon>0$ and any $\delta>0$, there exists some sufficiently
large $L_{\varepsilon,\delta}>0$ such that, if $mF(\tau)>1$,
\[
\p \biggl[ \sup_{l\ge0}\biggl\llvert \frac{W_{n+l+1}}{W_{n+l}}-
mF(\tau )\biggr\rrvert <\delta \Big| W_n\ge L_{\varepsilon,\delta}
\biggr]\ge 1-\varepsilon,
\]
and similarly, if $m(1-F(\theta))>1$,
\[
\p \biggl[ \sup_{l\ge0}\biggl\llvert \frac{S_{n+l+1}}{S_{n+l}}-
m\bigl(1-F(\theta )\bigr)\biggr\rrvert <\delta \Big| S_n\ge
L_{\varepsilon,\delta} \biggr]\ge 1-\varepsilon.
\]
\end{theor}

It should be noted that this result constitutes a new proof of
part~{(a)(ii)} of Theorem~\ref{th1}, and it will be used to deduce
part~{(a)(ii)} of Theorem~\ref{th2} in the next section. In order to
prove this result, we shall crucially make use of the following
Chernov-type estimates, which we prove first, based on Theorems~2 and~3
of~\citet{Coffman87}:

%
%le6.6 #&#
%
\begin{Lem}\label{lemcoff}
For any $\delta>0$, there exists constants $C,c>0$ such that, for all
$n,j\ge0$,
%
%e65 #&#
%
\begin{equation}
\p \biggl[ \biggl\llvert \frac{W_{n+1}}{W_n}-mF(\tau)\biggr\rrvert >\delta
\Big| W_n=j \biggr]\le Ce^{-cj},
\end{equation}
and similarly,
%
%e66 #&#
%
\begin{equation}
\p \biggl[ \biggl\llvert \frac{S_{n+1}}{S_n}-m\bigl(1-F(\theta)\bigr)\biggr
\rrvert >\delta \Big| S_n=j \biggr]\le Ce^{-cj}.
\end{equation}
\end{Lem}

\begin{pf}
By definition,
%
%e67 #&#
%
\begin{eqnarray}
&& \p \biggl[ \biggl\llvert \frac{W_{n+1}}{W_n}-mF(\tau)\biggr\rrvert >
\delta \Big| W_n=j \biggr]
\nonumber\\[-8pt]\\[-8pt]
&&\qquad =\p \biggl[\biggl\llvert
\frac{N_n(D_n(j),R_n(j))}{j}-mF(\tau )\biggr\rrvert >\delta \biggr].\nonumber
\end{eqnarray}
Let $\eta>0$. As $N_n:=N_n(\cdot, \cdot)$ is increasing in both
arguments, conditioning on $|R_n(j)/j-r|\le\eta$ (resp.,
$>\eta
$) and on $|D_n(j)/j-m|\le\eta$ (resp., $>\eta$) in the RHS, we obtain,
after several elementary transformations, that the latter rhs is
bounded above by
%
%e68 #&#
%
\begin{eqnarray}\label{eqchernpr}
&& \p \biggl[\biggl\llvert \frac{N_n(j(m+\eta),j(r+\eta))}j-mF(\tau)\biggr\rrvert >\delta
\biggr]\nonumber
\\
&&\quad {} +\p \biggl[\biggl\llvert \frac{N_n(j(m-\eta),j(r-\eta))}j-mF(\tau )\biggr\rrvert >\delta
\biggr]
\\
&&\quad {}+\p \biggl[\biggl\llvert \frac{D_n(j)}j-m\biggr\rrvert >\eta \biggr]+\p
\biggl[\biggl\llvert \frac{R_n(j)}j-r\biggr\rrvert >\eta \biggr].\nonumber
\end{eqnarray}
Choosing $\eta$ small enough such that $|(m+\eta)F(\tau((r+\eta
)/(m+\eta
)))-mF(\tau)|<\delta/2$, the first term of~(\ref{eqchernpr}) becomes
smaller than
%
%e69 #&#
%
\begin{eqnarray}
\label{eqboundpn}
&& \p \biggl[\biggl\llvert \frac{N_n(j(m+\eta),j(r+\eta))}j-(m+\eta)F \biggl(\tau
\biggl(\frac{r+\eta}{m+\eta} \biggr) \biggr)\biggr\rrvert >\frac\delta2 \biggr]
\nonumber\\[-8pt]\\[-8pt]
&&\qquad \leq  2 \exp \biggl(-\frac{\delta^2j}{16mF(\tau)+8\delta} \biggr),\nonumber
\end{eqnarray}
where the last inequality follows from Theorems~2 and~3 of \citet{Coffman87}. The second term of~(\ref{eqchernpr}) can be bounded similarly.

The two last terms also satisfy an exponential bound by Hoeffding's
inequality [see Theorem~2 of~\citet{Hoeff}], since the random variables
$D_n^k$ and $R_n^k$ are assumed to be bounded. This proves the result
in the weakest-first case.

As far as the strongest-first case is concerned, now using the
cap-unimodality as well as a simple bound for the minimum and the
maximum of $M_n$ over the respective intervals (see Lemma~\ref
{lemunimod}), we can deduce a bound similar to~(\ref{eqchernpr})
(with some additional terms), and the result will follow in an
analogous way.
\end{pf}

The following result constitutes a first step in the proof of
Theorem~\ref{thborneas}:

%pr6.7 #&#
%
\begin{Prop}\label{propborneas}
For any $\varepsilon>0$ and any $\delta>0$, there exists some sufficiently
large $L_{\varepsilon,\delta}>0$ such that, if $mF(\tau)>1$,
\[
\p \biggl[ \bigcap_{l\ge0} \biggl\{
\frac{W_{n+l+1}}{W_{n+l}}> mF(\tau )-\delta \biggr\}\Big| W_n\ge
L_{\varepsilon,\delta} \biggr]\ge 1-\varepsilon,
\]
and similarly, if $m(1-F(\theta))>1$,
\[
\p \biggl[ \bigcap_{l\ge0} \biggl\{
\frac{S_{n+l+1}}{S_{n+l}}> m\bigl(1-F(\theta)\bigr)-\delta \biggr\}\Big|
S_n\ge L_{\varepsilon,\delta
} \biggr]\ge1-\varepsilon.
\]
\end{Prop}

\begin{pf}
Choose $1<\tilde m<mF(\tau)$, and take $\delta=mF(\tau)-\tilde m>0$.
Then we can write the following inequalities, for any $L>0$ and any
$n\ge0$:
%
%e70 #&#
%e71 #&#
%e72 #&#
%
\begin{eqnarray}
&&\p \biggl[ \frac{W_{n+l+1}}{W_{n+l}}>\tilde m,\ \forall l\ge 0 \Big|
W_n\ge L \biggr]
\\
&&\qquad =1-\sum_{l=0}^\infty\p [
 W_{n+j+1}>\tilde mW_{n+j}
\nonumber\\[-8pt]\\[-8pt]
 &&\hspace*{75pt}\mbox { for all $0\le j<l$, and
}W_{n+l+1}\le\tilde mW_{n+l}\rrvert W_n\ge L ]\nonumber
\\
&&\qquad \geq 1-\sum_{l=0}^\infty\p \biggl[
 W_{n+j+1}>\tilde mW_{n+j}
\nonumber\\[-8pt]\\[-8pt]
&&\hspace*{77pt} \mbox{ for all $0\le j<l$, and }
\biggl\llvert \frac{W_{n+l+1}}{W_{n+l}}-mF(\tau )\biggr\rrvert \ge\delta\Big|
W_n\ge L \biggr],\hspace*{-25pt}\nonumber
\end{eqnarray}
since $\tilde m=mF(\tau)-\delta$. Now, given $W_n\ge L$, the
inequalities $W_{n+j+1}>\tilde mW_{n+j}$ for all $0\le j<l$ imply that
$W_{n+l}\ge\tilde m^lL$. Using this and dropping the intersection yields
%
%e73 #&#
%e74 #&#
%e75 #&#
%
\begin{eqnarray}
&&\p \biggl[ \frac{W_{n+l+1}}{W_{n+l}}>\tilde m,\ \forall l\ge 0 \Big|
W_n\ge L \biggr]
\\
&&\qquad \geq 1-\sum_{l=0}^\infty\p \biggl[
\biggl\llvert \frac
{W_{n+l+1}}{W_{n+l}}-mF(\tau)\biggr\rrvert \ge\delta \Big|
W_{n+l}\ge \tilde m^lL \biggr]
\\
&&\qquad \geq 1-\sum_{l=0}^\infty\sum
_{j=\tilde m^lL}^\infty\p \biggl[ \biggl\llvert
\frac{W_{n+l+1}}{W_{n+l}}-mF(\tau)\biggr\rrvert \ge\delta \Big| W_{n+l}=j
\biggr]
\nonumber\\[-8pt]\\[-8pt]
&&\qquad \ge1-C\sum_{l=0}^\infty\sum
_{j=\tilde
m^lL}^\infty e^{-cj},\nonumber
\end{eqnarray}
where the last inequality follows from the Chernov-type estimates given
in Lemma~\ref{lemcoff}. Now, since $\tilde m>1$, a straightforward
calculation yields
%
%e76 #&#
%e77 #&#
%
\begin{eqnarray}
\sum_{l=0}^\infty\sum
_{j=\tilde m^lL}^\infty e^{-cj}&=&\sum
_{l=0}^\infty e^{-c\tilde m^lL}\sum
_{j=0}^\infty e^{-cj}=\frac{1}{1-e^{-c}}\sum
_{l=0}^\infty e^{-c\tilde m^lL}
\\
&\leq& \frac{1}{1-e^{-c}}e^{-cL}\sum_{l=0}^\infty
e^{-c(\tilde m^l-1)}\le Ke^{-cL},
\end{eqnarray}
where the constant $K<\infty$ only depends on $c$ and on $\tilde m$.
Taking $L$ large enough thus yields the result. The same argument holds
in the strongest-first case.
\end{pf}

Now using Proposition~\ref{propborneas}, the proof of Theorem~\ref
{thborneas} becomes straightforward:

\begin{pf*}{Proof of Theorem~\ref{thborneas}}
Let $n$ be fixed and choose $0<\delta<mF(\tau)-1$. First consider the event
%
%e78 #&#
%
\begin{equation}
\Omega_{\varepsilon,\delta,L}= \biggl\{\frac{W_{n+l+1}}{W_{n+l}}> mF(\tau )-\delta\mbox{ for
all $l\ge0$, and }W_n\ge L \biggr\}\subset\Omega,
\end{equation}
for any $L>0$. Proposition~\ref{propborneas} means that $\p[\Omega
_{\varepsilon,\delta,L}]\ge(1-\varepsilon)\p[W_n\ge L]$ whenever $L\ge
L_{\varepsilon,\delta}$. On $\Omega_{\varepsilon,\delta,L}$, it is clear that
$W_{n+l}\to\infty$ a.s. as $l\to\infty$, and Lemma~\ref{lemme} then
implies that $W_{n+l+1}/W_{n+l}\to mF(\tau)$ a.s. as $l\to\infty$.
Hence, there exists a sufficiently large $K_{\varepsilon,\delta}\ge0$
such that
%
%e79 #&#
%
\begin{eqnarray}
\p \biggl[\Omega_{\varepsilon,\delta,L}\cap \biggl\{\sup_{l\ge
K_{\varepsilon,\delta}}\biggl
\llvert \frac{W_{n+l+1}}{W_{n+l}}-mF(\tau)\biggr\rrvert <\delta \biggr\} \biggr]&\geq& (1-
\varepsilon)\p[\Omega_{\varepsilon,\delta,L}].
\end{eqnarray}
Choose such a $K_{\varepsilon,\delta}$. Then, since, for any fixed $k$,
Lemma~\ref{lemme} gives $N_k(D_k(j),\break
 R_k(j))/j\to mF(\tau)$ a.s. as
$j\to
\infty$, we deduce that, for some sufficiently large $L'_{\varepsilon,\delta}>0$,
%
%e80 #&#
%
\begin{eqnarray}
\label{eqploplopced}
&& \p \biggl[ \sup_{0\le l< K_{\varepsilon,\delta}}\biggl\llvert
\frac
{W_{n+l+1}}{W_{n+l}}-mF(\tau)\biggr\rrvert <\delta\Big| W_{n+l}\ge
L'_{\varepsilon,\delta},\mbox{ for all $0\le l<K_{\varepsilon,\delta
}$}
\biggr]
\nonumber\\[-8pt]\\[-8pt]
&&\qquad \ge1-\varepsilon.\nonumber
\end{eqnarray}
Now note that, whenever $L\ge L'_{\varepsilon,\delta}$, we have on
$\Omega
_{\varepsilon,\delta,L}$ the inequality $W_{n+l}\ge L'_{\varepsilon,\delta}$
for all $0\le l<K_{\varepsilon,\delta}$, so that equation~(\ref
{eqploplopced}) yields, for any $L\ge L'_{\varepsilon,\delta}$,
%
%e81 #&#
%
\begin{eqnarray}
&& \p \biggl[\Omega_{\varepsilon,\delta,L}\cap \biggl\{\sup_{0\le l<
K_{\varepsilon,\delta}}\biggl
\llvert \frac{W_{n+l+1}}{W_{n+l}}-mF(\tau)\biggr\rrvert <\delta \biggr\} \biggr]
\nonumber\\[-8pt]\\[-8pt]
&&\qquad \ge(1-
\varepsilon)\p[\Omega_{\varepsilon,\delta,L}].\nonumber
\end{eqnarray}
We can thus conclude, for $L''_{\varepsilon,\delta}=\max(L_{\varepsilon,\delta
},L'_{\varepsilon,\delta})$, using, for instance, the Bonferroni inequality,
%
%e82 #&#
%e83 #&#
%
\begin{eqnarray}
&& \p \biggl[\sup_{l\ge0}\biggl\llvert \frac{W_{n+l+1}}{W_{n+l}}-mF(\tau
)\biggr\rrvert <\delta,\mbox{ and }W_n\ge L''_{\varepsilon,\delta}
\biggr]
\nonumber\\[-8pt]\\[-8pt]
&&\qquad \geq \bigl(2(1-\varepsilon)-1\bigr)\p[\Omega_{\varepsilon,\delta,L''_{\varepsilon,\delta
}}]\nonumber
\\
&&\qquad \geq (1-2\varepsilon) (1-\varepsilon)\p\bigl[W_n\ge L''_{\varepsilon,\delta}
\bigr].
\end{eqnarray}
Replacing $\varepsilon$ by $\varepsilon/3$ yields the result. The same
argument holds in the strongest-first case.
\end{pf*}

We can now turn to the proof of Theorem~\ref{thbounds}.

For that purpose, we first prove the following result, which is
interesting in itself:

%pr6.8 #&#
%
\begin{Prop}\label{propprepborn}
Let $(\Gamma_n)_n$ be any RDBP. Assume that $m(1-F(\theta))>1$. Then,
for any $\varepsilon,\delta>0$, there exists some sufficiently large
$L_{\varepsilon,\delta}>0$ such that, for all $n\ge0$,
\[
\p \biggl[ \bigcap_{l\ge0} \biggl\{
\frac{\Gamma
_{n+l+1}}{\Gamma
_{n+l}}\ge\frac{S_{n+l+1}}{S_{n+l}}-\delta \biggr\}\Big| \Gamma
_n=S_n\ge L_{\varepsilon,\delta} \biggr]\ge1-\varepsilon.
\]
\end{Prop}

\begin{pf}
Choose $\tilde m=m(1-F(\theta))-\delta/2$ and take $\delta$ small
enough so that $\tilde m>1$. Then we can write the following
inequalities, similarly as in the proof of Proposition~\ref{propborneas}:
%
%e84 #&#
%e85 #&#
%e86 #&#
%e87 #&#
%
\begin{eqnarray}
&&\p \biggl[ \frac{\Gamma_{n+l+1}}{\Gamma_{n+l}}>\tilde m,\ \forall l\ge0\Big|
\Gamma_n\ge L \biggr]
\\
&&\qquad = 1-\sum_{l=0}^\infty\p [
\Gamma_{n+j+1}>\tilde m\Gamma _{n+j}
\nonumber\\[-8pt]\\[-8pt]
&&\hspace*{75pt} \mbox{ for all $0\le j<l$,
and }\Gamma_{n+l+1}\le\tilde m\Gamma _{n+l}\rrvert \Gamma_n\ge L ]\nonumber
\\
&&\qquad \geq 1-\sum_{l=0}^\infty\p \bigl[
\Gamma_{n+l+1}\le\tilde m\Gamma _{n+l}\rrvert
\Gamma_{n+l}\ge\tilde m^lL \bigr]
\\
&&\qquad \geq 1-\sum_{l=0}^\infty\sum
_{j=\tilde m^lL}^\infty\p [ \Gamma _{n+l+1}\le
\tilde m\Gamma_{n+l}\rrvert \Gamma_{n+l}=j ].
\end{eqnarray}
Now note that, given $\Gamma_{n+l}=S_{n+l}=j$, we have
\[
S_{n+l+1}=M_{n+l}\bigl(D_{n+l}(j),R_{n+l}(j)
\bigr)\le Q^\Gamma _{n+l}\bigl(D_{n+l}(j),R_{n+l}(j)
\bigr)\le\Gamma_{n+l+1}.
\]
Using this and the Chernov-type estimates given in Lemma~\ref
{lemcoff}, we see that for any $\varepsilon>0$, we can find a
sufficiently large $L_{\varepsilon,\delta}>0$ such that
%
%e88 #&#
%e89 #&#
%
\begin{eqnarray}
&& \p \biggl[ \frac{\Gamma_{n+l+1}}{\Gamma_{n+l}}>\tilde m,\ \forall l\ge 0\Big|
\Gamma_n\ge L_{\varepsilon,\delta} \biggr]
\nonumber\\[-8pt]\\[-8pt]
&&\qquad \geq  1-\sum
_{l=0}^\infty \sum_{j=\tilde m^lL}^\infty
\p [ S_{n+l+1}\le\tilde mS_{n+l}\rrvert
S_{n+l}=j ]\nonumber
\\
&&\qquad \geq 1-C\sum_{l=0}^\infty\sum
_{j=\tilde m^lL_{\varepsilon,\delta
}}^\infty e^{-cj}\ge1-\varepsilon,
\end{eqnarray}
where we argued in the second inequality exactly as in the proof of
Proposition~\ref{propborneas}. Putting this together with
Theorem~\ref{thborneas} yields the result.
\end{pf}

We can now immediately deduce Theorem~\ref{thbounds}:

\begin{pf*}{Proof of Theorem~\ref{thbounds}}
On the one hand, if $m(1-F(\theta))>1$, Proposition~\ref
{propprepborn} yields
%
%e90 #&#
%
\begin{equation}
\label{eqborneprep1} \lim_{L\to\infty}\p \Bigl[\lim_{n\to\infty}
\Gamma_n(L)\ge\lim_{n\to\infty
}S_n(L)
\Bigr]=1.
\end{equation}
On the other hand, if $m(1-F(\theta))<1$, then we have proven that
$S_n(L)\to0$ almost surely for any $L>0$ [see Theorem~\ref{th2}(a)(i)],
and thus equation~(\ref{eqborneprep1}) holds trivially. Finally, using
Proposition~\ref{propbounds} for the upper bound for $(\Gamma_n)_n$
yields, as desired,
%
%e91 #&#
%
\begin{equation}
\label{eqborneprep2} \lim_{L\to\infty}\p \Bigl[\lim_{n\to\infty}S_n(L)
\le\lim_{n\to
\infty
}\Gamma_n(L)\le\lim
_{n\to\infty}W_n(L) \Bigr]=1.
\end{equation}
It remains to prove that $q_W=1 \Rightarrow  q_\Gamma=1
\Rightarrow
q_S=1$. The first implication simply follows from Proposition~\ref
{propbounds}. Now, assuming $q_\Gamma=1$, we easily deduce, using the
regularity assumption~(iii) of Section~\ref{chapnatregcond}, that, for
any $L$, we have \mbox{$\Gamma_n(L)\to0$} a.s., as $n\to\infty$. Hence,
equation~(\ref{eqborneprep2}) gives $\p[\lim_nS_n(L)=0]\to1$, as
$L\to
\infty$. It follows then that $q_S=1$, as we can show by contradiction,
using again the regularity assumption (iii).
\end{pf*}

%s6.8 #&#
\subsection{Extinction criterion for the s.f.-society (second part)}\label{proofth2+}
In this section, we finish the proof of Theorem~\ref{th2}, using the
results of previous section (in particular, Theorem~\ref{thborneas}).
Therefore, we will need the boundedness assumption for all the random
variables $D_n^k$, $X_n^k$ and $R_n^k$ [see regularity assumption (v)
of Section~\ref{chapnatregcond}], while only the boundedness of the
$X_n^k$'s was needed for the first part of the proof in Section~\ref
{proofth2}.

\begin{pf*}{Proof of Theorem~\ref{th2}}
The first part of Theorem~\ref{th2}(a) has already been proven in
Section~\ref{proofth2}.

Now look at the second part, assuming $r\le m\mu$ and $m(1-F(\theta
))>1$. Then, Theorem~\ref{thborneas} (or Proposition~\ref
{propborneas}) gives, for any $\varepsilon>0$ and for $\delta>0$ with
$m(1-F(\theta))>1+\delta$,
%
%e92 #&#
%
\begin{eqnarray}
\qquad 1-q_S&\geq& \p \biggl[\bigcap_{l\ge0} \biggl
\{\frac
{S_{n+l+1}}{S_{n+l}}>m\bigl(1-F(\theta)\bigr)-\delta \biggr\} \biggr]\ge (1-
\varepsilon )\p[S_n\ge L_{\varepsilon,\delta}],
\end{eqnarray}
so that, using the regularity assumption (iii) of Section~\ref
{chapnatregcond}, we can deduce $q_S<1$.

Finally, equation~(\ref{eqsupbornes}) of Theorem~\ref{th2} is another
direct consequence of Theorem~\ref{thborneas} (or again
Proposition~\ref{propborneas}).
\end{pf*}

%s6.8.1 #&#
\subsubsection{Alternative criterion for the s.f.-society}

For human societies\break (which constitute here the main focus interest) the
condition that all random variables are bounded
can be well defended. Viewing applications of our results for
populations other than human populations, it may be desirable to do
without the boundedness assumption. Recall that we needed this
assumption only in the proof that the s.f.-process may survive if
$m(1-F(\theta))>1$. Assume $r\le m\mu$ and $\alpha:=m(1-F(\theta))>1$.
Recall also that the condition of bounded claims must be maintained for
the extinction criterion for the s.f.-process. For the other variables we
have, however, an alternative condition:

%le6.9 #&#
%
\begin{Lem}\label{lemstochinc} Suppose that the sequence
$(M(D(t),R(t))_t$ is stochastically increasing in $t$ for all $t $
sufficiently large.
Further, let $\theta$ be as defined in (\ref{thetaequ}). Then
\[
m\bigl(1-F(\theta)\bigr)>1\quad\implies\quad q_S<1.
\]
\end{Lem}

\begin{pf}
If $M(D(t),R(t))_t$ is stochastically increasing in $t$ for all $t$
sufficiently large, then there exists an integer $t_0$, say, such that
for all $v\ge0$ and for all $t\ge t_0$,
\[
\p\bigl[ M\bigl(D(t),R(t)\bigr)\ge v\bigr] \le\p\bigl[M\bigl(D(t+1),R(t+1)\bigr)
\ge v\bigr].
\]
Hence, for $t_1, t, t_2 \in\N$ with $t_0\le t_1\le t \le t_2$, we
deduce, for all $v\ge0$,
\[
\max_{t\in[t_1, t_2]} \p\bigl[ M\bigl(D(t),R(t)\bigr)\ge v\bigr]=\p
\bigl[ M\bigl(D(t_2),R(t_2)\bigr)\ge v\bigr].
\]
The proof of Theorem~\ref{th1}(a)(ii) for the w.f.-process can now be
adapted immediately to the s.f.-process by replacing $mF(\tau)$ by
$m(1-F(\theta))$. Indeed, for $j$ sufficiently large the maximum
probability is always on the right border of the corresponding interval
and thus under control. Inequality~(\ref{inequ}) now remains true for
the s.f.-process as well, and the rest of the proof can be rewritten
accordingly.
\end{pf}

%re6.10 #&#
%
\begin{rem}
We do not know whether $(M(D(t), R(t)))_t$ is stochastically increasing
for sufficiently large $t$ for all choices of distributions of $D_n^k$
and $R_n^k$ with finite second moments. A beginning argument is as follows.
First note that $D(t)$ and $R(t)$ are stochastically increasing in $t$
and that, as we also know,
%
%e93 #&#
%
\begin{equation}
\frac{M(D(t),R(t))}{t} \to m\bigl(1-F(\theta)\bigr)=:\alpha\qquad\mbox{a.s. as }t\to
\infty,
\end{equation}
and the convergence of the expectations $\E [M(D(t),R(t))
]/t$ thus holds as well. Thus we can find $\varepsilon>0$ with $\alpha
-\varepsilon>1$, such that $\E [M(D(t),R(t)) ]$ is increasing
for all $t\ge t_0(\varepsilon)$. This implies that, for all $t \ge
t_0(\varepsilon)$,
%
%e94 #&#
%e95 #&#
%
\begin{eqnarray}
\E \bigl[M\bigl(D(t),R(t)\bigr) \bigr]&=&\sum_{v=0}^\infty
\p \bigl[M\bigl(D(t),R(t)\bigr)\ge v \bigr]
\\
&\leq& \sum_{v=0}^\infty\p \bigl[M
\bigl(D(t+1),R(t+1)\bigr)\ge v \bigr]
\nonumber\\[-8pt]\\[-8pt]
&=& \E \bigl[ M\bigl(D(t+1),R(t+1)\bigr) \bigr].\nonumber
\end{eqnarray}
It would be natural to believe that this inequality holds not only for
the sums, but also for the corresponding terms of the sums, from some
$t$ onward, which would mean that $M$ is stochastically increasing.

Note that $M$ does actually not need to be strictly stochastically
increasing for the proof of Lemma~\ref{lemstochinc} to hold: some
rapidly decreasing error (in $v$ and $t$) can indeed be admitted. More
precisely, it suffices to show that there is some big $t_0$ such that,
for all $t\ge t_0$ and for all $s$,
%
%e96 #&#
%
\begin{equation}
\p \bigl[M\bigl(D(t),R(t)\bigr)\ge s \bigr]\le\p \bigl[M\bigl(D(t+1),R(t+1)
\bigr)\ge s \bigr]+C(t,s),
\end{equation}
where the corrections $C$ must be not too big, in the sense that $\sum_jC(\alpha_\varepsilon^{j-1}s,\break \alpha_\varepsilon^js)<\infty$, where we choose
$\alpha_\varepsilon=\alpha-\varepsilon>1$, for some $\varepsilon>0$ small enough.

However, even such a weakened form seems hard to prove, and we let this
problem as an open question.
\end{rem}

%s7 #&#
\section{Significance of the results}
The fact that the survival criteria for both extreme societies can be
given explicitly makes the Envelopment theorem significant. We first
note that these theorems give extinction/survival criteria in terms of
the parameters $m$ (mean offspring number), $r$ (mean resource
creation) and the distribution function of resource claims $F$ (from
which we also know the mean resource claim $\mu$). Interestingly, in
each case the solution of a last relevant parameter ($\tau$ and
$\theta
$, resp.) is obtained by solving a simple integral equation involving
the \textit{Lorenz curve} known from Economics. Thus the critical
boundaries are explicit.

Now recall the ``safe-haven'' property. We have seen that if the
survival probability of the w.f.-society is strictly positive, then,
however small it may be when starting with few individuals, it
converges quickly to $1$ with increasing size.
We conclude that, provided $q_{W}<1$, any society has always the option
of a very probable survival by letting converge their rules, if
necessary, toward the rules of the w.f.-society. If $q_{W}=1$, however,
then, with a fixed offspring probability law $(p_k)_k$, society must
draw the consequences, because, viewing the chance of survival, there
is no alternative. The individuals live beyond their means and must be
instructed by the society to either become more modest in average
claims of resources or else to increase the average reproduction of resources.

No other society in this model does as much for ensuring survival as
the w.f.-society. The price to pay under the same fixed distribution is
the most modest standard of living of individuals in this society.

The s.f.-society constitutes the other extreme. Under the given
assumptions this society does the most for the standard of living of
the few. However, it jeopardizes the prospects of survival more than
any other society.

Both \textit{extreme} societies form an {envelope} for any society in the
sense that, in the long run, no society can exceed these bounds. We may
call it a \textit{quasi-envelope} because the w.f.-society leads to a
definite uniform upper bound process whereas the s.f.-society leads,
strictly speaking, only to a very probable lower bound process.
However, we know from the conditional envelopment theorem that in the
long run there cannot exist a strictly better lower bound process, so
that it is not a misnomer to speak of an envelope rather than of a
quasi-envelope.

%s7.1 #&#
\subsection{Tractability of the model}
Clearly, RDBPs are still relatively simple models compared with what we
expect we would need to model societies
in a most realistic way. However, there are strong reasons why they
should earn our attention.

First, of course, it is not realistic to look for a perfect model, and,
keeping this in mind, RDBPs seem to be a good approach because they
give considerable room for modeling aspects.

Second, RDBPs yield, as we have seen, not only the envelopment theorem
for societies but also explicit survival criteria in form of
quantifiable critical relationships between society forms. This fact
should not be taken for granted. As we have seen in the general
definition of RDBPs, realistic society forms will typically impose
complicated structures. Almost all interesting forms are too
complicated to be tackled by generating functions or martingale
arguments, the most powerful tools in branching process theory.

Third, RDBPs are remarkably robust. The main results flowing from them
hold in more general settings.
So, in particular, consider the assumption that reproduction within a
RDBP is asexual. It is interesting to know what happens if we replace
this assumption by the natural assumption that reproduction depends on
two sexes [see~\citet{Daley86}; see also~\citet{Molina10} for a review of
known results in this domain].

The answer is in fact in favor of RDBPs. Since survival is only
possible if the RDBP can grow without limits, the asexual reproduction
mean $m$ can here be substituted by the so-called limiting \textit{average
reproduction mean.} The average reproduction mean for a total of $k$
``mating units'' $m(k)$ is defined in equation~(1) of~\citet{Bruss84}. In
the notation of the present paper, it translates into
%
%e97 #&#
%
\begin{equation}
m(k)=\frac{1}k\E \bigl[ \widetilde D_n(\widetilde
\Gamma_n)\rrvert \widetilde \Gamma_n=k \bigr],
\end{equation}
where, unlike $(\Gamma_n)_n$, the modified $(\widetilde\Gamma_n)_n$ counts
now the number of ``mating units'' (and not individuals) present in
generation $n$, and $\widetilde D_n(\widetilde\Gamma_n)$ denotes the number of
mating units generated by these for the next generation. If $\ell:=\lim_{k\to\infty} m(k)$ exists (which is the case for the majority of
natural mating functions), then the specific form of the mating
function becomes irrelevant for extinction criteria as soon as the
population size has become sufficiently large. Hence the generalization
to sexual reproduction may affect the initial chances of reaching
larger numbers of individuals within a RDBP but does not affect the
main results.

Similarly, a little reflection shows that passing from discrete time
generations to more realistic ``moving'' generations makes it
technically harder to define the precise meaning of the strings of
resource claims. However, under some reasonable conditions, there are
ways around the formal problems via discretization, and moving
generations do not impair the essence of the found critical
relationships between the society form, the parameters $\tilde m$, $r$,
$\mu$ and the function $F$.
%s8 #&#
\section{Conclusions}
Returning to the RDBPs we have defined, we shall comment for the
remainder of this paper on real-world conclusions by confining our
interest to the important ones.

On the one hand, we have some intuition that all societies we may think
of should have, for fixed probability laws of natality, of resource
production, and of resource consumption, somewhere their limits. On the
other hand, as we have seen, this intuition is partially wrong and
requires a thorough revision. Rigorous arguments then helped to
overcome the new difficulties. These arguments lead to more subtle
conclusions. The more remarkable is, in our opinion, that, after
refinement, a major part of the original intuition is now proven true.

It is tempting to apply these results by looking at society forms that
we see around us, or at those that mankind has tried in the past. Much
insight may be gained from learning why certain society forms have
failed, and why others seem to do, or to have done, relatively well. It
would be nice to see that scientists who have access to data or
estimates needed for the analysis presented here will find such
questions a real challenge.

In the following, we shortly discuss the main features of a few
selected societies, and how they can be seen as RDBPs.

%s8.1 #&#
\subsection{A brief comparison of major society forms}
%s8.1.1 #&#
\subsubsection{Mercantilism}
Mercantilism was the dominant policy for western societies for most of
the 16th century up to the end of the 18th century, and in some
countries even to the beginning of the 19th century. There are several
forms of mercantilism, but with respect to RDBPs, there is a common
denominator to all different forms of mercantilism, that is, as we
shall argue, a state-controlled ``head-and-tail policy'' of distributing
resources.

The philosophy of mercantilism is that the wealth of a nation, compared
with the wealth of other nations, is a zero-sum game. This implies that
leaders concentrate their interest on the competition between different
states for a common fixed wealth of the world. At the same time, the
idea behind was also that a rich country can afford a strong army to
defend wealth. Strict mercantilists, exemplified by Colbert in France,
concluded that all what counts is getting the wealth into the own
country by exports and keeping production costs of goods as low as
possible. A~positive balance of trades was the main concern; imports
were highly taxed.

In the interpretation of RDBPs, the members of the government or
kingdom as well as rich merchants are typically those with the larger
claims. They form the ``head,'' but they are relatively small in number.
The ``tail'' consists of the claims of those individuals in the
population who have to produce (farmers, workers, etc.). In agreement
with the philosophy of mercantilism, money which was spent was seen as
lost. Hence production costs were kept as low as possible. The right of
emigration was denied in some countries during certain periods, and in
such cases unsatisfied claims must be interpreted as removal by death
in RDBPs. The tail is made of the modest claims, but now to be
satisfied from a reduced resource space.

In words of RDBPs mercantilism may be seen as a hybrid society, one
part living under a s.f.-policy the other part (the last majority of
people) under an enforced w.f.-policy. It can be modeled as the ``sum'' of
two RDBPs, by assuming (e.g.) that a certain percentage $p$ of
the common resource space is reserved for the rich and $1-p$ for the poor.

%s8.1.2 #&#
\subsubsection{Enlightened mercantilism}
As we understand today, mercantilism suffered from a lack of
experience, or, as critics would say, a lack of understanding. The
rules of import and export were very strict, and the idea that free
trade creates value and that wealth is by far not a zero-sum game had
to await the arrival of great economists.

Although economists like Dudley North, John Locke and Adam Smith in
particular, undermined much of mercantilism, and saw themselves as
anti-mercantilists, one may call them today \textit{enlightened}
mercantilists. The new ideas of the great value of free trade and of a
motivated work force seemed almost revolutionary, and it is true that
they brought a great change and opened the way to more recent economic
societies. However, at the beginning, a part of the philosophy of
mercantilism was still in force. So, for instance, the priority of
inexpensive production before fair compensation for work was still
rather present, and not all defendants of the new ideas accepted
already the conclusion that the wealth of nations is no zero-sum game.

The graph of claims in a society under enlightened mercantilism,
modeled by an RDBP, would therefore still resemble that of
mercantilism, and we would suggest a similar approach to model it as a
hybrid w.f.-society and s.f.-society. There are two clear differences:
first, nations became richer once the barriers to free trade imposed by
the classical mercantilism were relaxed. Hence, the available resource
space $s$ should be seen as getting significantly larger. Second, the
barrier between small claims and large claims became much more fluent
because the poor ones could escape more easily the society into which
they were born.

%s8.1.3 #&#
\subsubsection{Modern western societies}
As much as we might see large differences between western nations and
their society forms, we are reminded that everything is just relative.
Seen as RDBPs with their graph of resource claims, these societies are
not really that different. On the one hand many have quite an important
common denominator, namely a rather solid social security system. On
the other hand, they typically display society forms which leave quite
some room for individual freedom and performance. They are ranging from
(controlled) free societies up to a laissez-faire societies. The first
feature (social security systems) implies that the smallest claims are,
compared with the larger ones, much bigger than in mercantilist
societies, and the second one that the larger claims go in all of them
through a wide range. It is this mixture of a laissez-faire society
behavior as well as \mbox{w.f.-}features and s.f.-features, which would make it
hard to model western societies by RDBPs in such a way that finer
differences between modern western societies show up clearly.

%s8.1.4 #&#
\subsubsection{Capitalism}
Concerning capitalism, RDBPs can tell much more. As indicated before,
the s.f.-society shares features of an extreme form of capitalism, and it
is easy to understand why. The larger claims stem from the class of
people who have the power to defend them within society and are thus
strongly correlated with financial power. The graph of claims would
look similar to the one for the s.f.-society.

In one point, things are of course different. It suffices to ask who
would effectively still create the resources which are made available
by manual work. Even in extreme forms of capitalism, there must still
be some space for the lower claims. Still, the s.f.-process is probably
the most convincing RDBP-approach to modeling capitalism.

%s8.1.5 #&#
\subsubsection{Communism}
Correspondingly, the w.f.-society shares undeniably features with an
extreme form of communism. If means of production are completely in the
hands of the people's society, the larger claims are no longer backed
by personal power due to individual ownership. Hence, the smaller
claims, backed by the masses, are automatically given priority.

Again, exception must be made for the top political class, but in
comparison to the whole population the number of individuals in this
class is very small. Hence the overall picture of the resource claim
distribution over the whole space of currently available resources
would have much similarity with the one of the \mbox{w.f.-}society, and the
corresponding RDBP is  well-understood.

%s8.1.6 #&#
\subsubsection{Extreme communism versus extreme capitalism}
If we accept the comparison of the w.f.-society and the s.f.-society with
extreme versions of communism and capitalism, respectively, then these
two societies, seen as RDBPs form, as we have seen, an envelope for all
society forms.

It is a strange envelope indeed: namely, its boundaries attract and
repulse individuals at the same time. Neither of the extreme societies
can be expected to be seen by individuals as attractive, at least not
if resources are limited. People want to get away from extreme
communism by Hypothesis~\ref{hyp2}. They would like to increase their standard
of living as far as possible and thus leave extreme communism as
quickly as possible.

In the extreme form of capitalism, things are equally unstable.
Neglecting the case of arbitrarily many resources for everybody, the
greed of the strongest ones kills the population from above because,
typically, a large proportion of the people will leave. This society
will soon have to revise thoroughly its philosophy. One reason is, as
shown in this article, the necessity to have more people to increase
the survival probability. Another reason is that there must be enough
people to make the present resources available.

It is difficult to imagine that extreme capitalism could survive in
reality. In our RDBP model, it is possible under the condition
$m(1-F(\theta))>1$. This imposes not only a high productivity, but also
that individual resource creations stay i.i.d. random variables.
However, work forces from different society classes are in reality not
easily exchangeable.

%s8.1.7 #&#
\subsubsection{Marxism}
Returning to Hypothesis~\ref{hyp2} and our conclusion that extreme communism
will never last long, RDBPs give, in a certain way, reason to Marx for
the facts, although not for the conclusion. Marx saw that it may be
hard to convince people of the advantages of his new ideology. He
believed that the ``Mehrwert'' (added value) of labor would then mainly
benefit the working class who \textit{creates} value, and that, as a
result, income of the lowest class would go up. In RDBPs, this \textit{does not} happen unless the mean production per individual would go up.

One has to give to Marx that he was hoping that this would be the case,
that is, that not only the income of the lower class but also the
overall productivity would go up. Many opponents of communism are
convinced of the very contrary, namely that the loss of personal
advantages in the ideal communistic society will bring productivity down.

We have to stay neutral on the latter discussion because our analysis
touches only the critical relationship between demand and productivity,
and not what is behind. What RDBPs show [compare with example~(i) in
Section~\ref{exs}] is that the added value would have to be
considerable to convince people that their search for a higher standard
of living would be of little importance.

When Marx formulated his famous promise ``From each according to his
ability, to each according to his needs,'' then there is nothing
remarkable about this in a RDBP. If a low productivity of an individual
comes together with a large claim, then it must be expected to happen
in a model with i.i.d. random variables. Any society can propose to
serve claims independently of their size, provided that the resource
space allows for it. But this is the point. To know whether the
resource space will allow for the advertised generosity, Marx would
face the crucial question of productivity.

%s8.1.8 #&#
\subsubsection{Leninism}
Coming back to the question of added value, Lenin was, compared to
Marx, much more radical in his views. He went so far as advocating the
new ideas to be imposed by a vanguard party approach, that is, a
revolution. And this is what he got. One serious question should be
asked here. Why did Lenin push toward a revolution before knowing more
about the alleged added value by moving production tools in the hands
and property of the working class? This was completely new territory;
there was no prior experience, there were no data. What does not work
in theory cannot work in practice. For RDBPs at least it cannot work in
theory if there is no added value in creating resources.
%A revolution is a priori more likely to eat up resources than create
%resources,
And it is difficult to imagine that Lenin could possibly have known
more about the added value in Marx's ``model'' than Marx knew himself.

%s8.2 #&#
\subsection{Comparing ideologies}
As we have seen, when maintaining the mean resource creation constant,
extreme communism cannot be a stable society. Nevertheless, it has one
\textit{undeniable and distinguished} advantage worth repeating: with
fixed resource creation (productivity), it is superior to any other
society with respect to survival probability. The price to pay is a low
average standard of living. Extreme capitalism leads to the highest
possible average standard of living. However this does not mean much in
a society where so many of the next generation may have to emigrate.
Note that in our RDBP-model this fate may hit equally likely the own
descendants so that the conclusions may pass the test of ``political
correctness'' so frequently evoked today.

Individuals typically wish to increase their standard of living and
Hypothesis~\ref{hyp2} will push extreme communism more and more into
consumption, that is, into a ``stronger'' first society of some kind. If
this move stays without control, then it will approach capitalism, the
declared enemy of communism.

No wonder capitalism and communism are declared enemies. There is more
to it than politics, more than just a great difference of ideology.
Seen as RDBPs, the difference between them is fundamental because they
are as opposed to each other with respect to the natural Hypotheses~\ref{hyp1}~and~\ref{hyp2}
as they possibly can be. The extreme communism is doing
everything for Hypothesis~\ref{hyp1} and nothing for Hypothesis~\ref{hyp2}, and the
extreme capitalism does exactly the opposite. In the setting of RDBPs,
both opponents are born losers, as we have seen, and there is little
reason for born losers to be declared enemies. Societies may test these
limits, as they have done, but any attractive society is bound to be
somewhere else within the envelope of societies.

%s8.3 #&#
\subsection{A modest outlook}
If mankind has seen several other variations of societies which seem
more attractive than the extreme societies we had to discuss before
because they stood out, then this is due to an increase of
understanding. Understanding the mechanisms and the impact of financial
instruments of a national economy, including all the instruments of
monetary policy and labor policy, these are essential factors as we
believe today, and there is little reason to doubt this. However, there
is also little reason to ignore the message of RDBPs. If the model is
accepted, then the interesting move of any society will always be
toward something lying between the w.f.-society and the s.f.-society. Not
accepting this conclusion means not agreeing with Hypotheses \ref{hyp1} and~\ref{hyp2}.
As far as we are aware, nothing definite as this has ever been stated
and proven before.

Could data help mankind converge in its search for an optimal policy
toward an ``optimum?'' Hopefully yes, at least toward something like an
optimum region. Nevertheless, Hypotheses \ref{hyp1} and~\ref{hyp2} will always be in
force, and hence fluctuations of policies, even around an alleged
optimum must be expected to be part of the game.

To know to what extent, this is much harder. As we understand, \textit{dynamical} policies should now be defined and studied,
whereas this paper studies the development of a society by control decisions based on \textit{static} RDBPs. The crucial question is then how to model
\textit{dynamical} policies in such a way that the tractability of the
global model (already an important issue in the present paper) can be maintained. We hope that this article may
help to pave the way for further research.

\section*{Acknowledgments}
The authors are grateful to the referee, to the Associate
Editor and to the Editor for many critical
remarks and stimulating comments.

%suskaldyti doi

% imsref loaded by linak, 2014-03-26 14:26:02
% imsref loaded by linak, 2014-03-27 16:26:02

\printaddresses

\end{document}